\documentclass[reqno]{amsart}
\usepackage[utf8x]{inputenc}  

\usepackage{setspace}
\usepackage[left=3cm, right=3cm, bottom=4cm]{geometry}                 

\usepackage{graphicx} 
\usepackage{pgf,tikz}
\usetikzlibrary{arrows}
\usepackage{amssymb}
\usepackage{pdfsync}
\usepackage{mathrsfs}
\usepackage{hyperref} 
\usepackage{verbatim} 
\usepackage{epstopdf}
\usepackage{bbm} 
\usepackage[colorinlistoftodos,prependcaption,textsize=tiny]{todonotes}
\usepackage{xargs}
\usepackage{bm}

\def\R{\mathbb{R}}

\def\N{\mathbb{N}}
\def\Z{\mathbb{Z}}
\def\C{\mathbb{C}}

\def\supp{{\rm supp}}

\renewcommand{\d}{\text{\rm d}}

\newcommand{\mc}{\mathcal}

\newtheorem{theorem}{Theorem}
\newtheorem{conjecture}[theorem]{Conjecture}

\newtheorem{corollary}[theorem]{Corollary}

\newtheorem*{definition*}{Definition}

\newtheorem{proposition}[theorem]{Proposition}
\newtheorem{lemma}[theorem]{Lemma}

\makeatletter
\DeclareFontFamily{U}{tipa}{}
\DeclareFontShape{U}{tipa}{m}{n}{<->tipa10}{}
\newcommand{\arc@char}{{\usefont{U}{tipa}{m}{n}\symbol{62}}}%

\makeatother

\numberwithin{equation}{section}

\allowdisplaybreaks

\newcommand{\intav}[1]{\mathchoice {\mathop{\vrule width 6pt height 3 pt depth  -2.5pt
\kern -8pt \intop}\nolimits_{\kern -6pt#1}} {\mathop{\vrule width
5pt height 3  pt depth -2.6pt \kern -6pt \intop}\nolimits_{#1}}
{\mathop{\vrule width 5pt height 3 pt depth -2.6pt \kern -6pt
\intop}\nolimits_{#1}} {\mathop{\vrule width 5pt height 3 pt depth
-2.6pt \kern -6pt \intop}\nolimits_{#1}}}

\newcommand{\intavl}[1]{\mathchoice {\mathop{\vrule width 6pt height 3 pt depth  -2.5pt
\kern -8pt \intop}\limits_{\kern -6pt#1}} {\mathop{\vrule width 5pt
height 3  pt depth -2.6pt \kern -6pt \intop}\nolimits_{#1}}
{\mathop{\vrule width 5pt height 3 pt depth -2.6pt \kern -6pt
\intop}\nolimits_{#1}} {\mathop{\vrule width 5pt height 3 pt depth
-2.6pt \kern -6pt \intop}\nolimits_{#1}}}

\title[Fourier optimization]{Fourier optimization, the least quadratic non-residue, and the least prime in an arithmetic progression}

\author[Carneiro]{Emanuel Carneiro}
\author[Milinovich]{Micah B. Milinovich}
\author[Quesada-Herrera]{Emily Quesada-Herrera}
\author[Ramos]{Antonio Pedro Ramos}

\address{
ICTP - The Abdus Salam International Centre for Theoretical Physics, 
Strada Costiera, 11, I - 34151, Trieste, Italy.}
\email{carneiro@ictp.it}

\address{Department of Mathematics, University of Mississippi, University, MS 38677 USA.}
\email{mbmilino@olemiss.edu}

\address{University of Lethbridge, Department of Mathematics and Computer Science, 4401 University Dr W, Lethbridge, AB T1K 3M4, Canada}
\email{emily.quesadaherrera@uleth.ca}

\address{IMPA - Instituto de Matem\'atica Pura e Aplicada, Rio de Janeiro - RJ, Brazil, 22460-320.}
\email{antonio.ramos@impa.br}

\date{\today}

\makeatletter
\@namedef{subjclassname@2020}{\textup{2020} Mathematics Subject Classification}
\makeatother
                                        
\begin{document}

\begin{abstract} By means of a Fourier optimization framework, we improve the current asymptotic bounds under GRH for two classical problems in number theory: the problem of estimating the least quadratic non-residue modulo a prime, and the problem of estimating the least prime in an arithmetic progression.

\end{abstract}

\maketitle 

\section{Introduction} In this paper, we apply a Fourier optimization framework to study two classical problems in number theory: (i) the problem of estimating the least quadratic non-residue modulo a prime, and (ii) the problem of estimating the least prime in an arithmetic progression. We are interested in establishing the strongest possible asymptotic bounds under the assumption of the Generalized Riemann Hypothesis (GRH) for Dirichlet $L$-functions. Our results are inspired by the previous work of Lamzouri, Li, and Soundararajan \cite{LLS}, who studied the same problems, and of Carneiro, Milinovich, and Soundararajan \cite{CMS}, who developed a Fourier optimization framework to study the maximum gap between primes assuming the Riemann Hypothesis. We provide improvements over the results in \cite{LLS}, both conceptually and numerically. One of the key elements of our approach is a thorough study of the extremal problems in Fourier analysis associated to these number theory problems. Throughout the paper we adopt the following normalization for the Fourier transform of a function $F \in L^1(\R)$:
\[
\widehat{F}(t) := \int_{-\infty}^{\infty} e^{-2 \pi i x t }\,F(x)\,\d x.
\]

\subsection{The least quadratic non-residue} \label{LQNR_Sec_11} Let $p$ be an odd prime, and let $n_p$ denote the least quadratic non-residue modulo $p$. As a consequence of the P\'olya-Vinogradov inequality, I.~M.~Vinogradov developed a clever trick  that established the bound $n_p \ll p^{\frac{1}{2\sqrt{e}}}\log^2 p$. Burgess later combined Vinogradov's trick with his bounds for character sums to show that $n_p = O_\varepsilon\big(p^{\frac{1}{4\sqrt{e}}+\varepsilon}\big)$ for all $\varepsilon>0$. Vinogradov further conjectured that $n_p \ll_\varepsilon p^\varepsilon$ for all $\varepsilon>0$, and Vinogradov's conjecture is now known to follow from the generalized Lindel\"of hypothesis and therefore from GRH. 

\smallskip

Under GRH, Ankeny \cite{An} proved that $n_{p}  \ll \log^2p$. The order of magnitude of this estimate has never been improved. Bach \cite{Bach} later established the asymptotic estimate $n_p \leq (1 + o(1))\log^2p$ which was improved by Lamzouri, Li, and Soundararajan \cite{LLS, LLS2} to $n_{p} \leq (0.794 + o(1))\log^2p$. 

\smallskip

We consider here the following constant arising from a Fourier optimization problem:
\begin{align}\label{20240403_11:21}
 \mc{C}(1) := \sup_{0 \neq F \in \mc{A}} \frac{2\pi}{\|F\|_1}\left( \int_{-\infty}^{0} \widehat{F}(t)\,e^{\pi t} \,\d t - \int_{0}^{\infty} \big|\widehat{F}(t)\big|\,e^{\pi t} \,\d t \right),
\end{align}
where the supremum is taken over the class of functions $\mc{A} = \{F:\R \to \C\,; \ F \in L^1(\R) \,;\ \widehat{F} \ {\rm is\ real \mbox{-}valued}\}$. We improve the current best conditional bounds for the least quadratic non-residue by establishing the following connection.

\begin{theorem}\label{Thm1_LQNR}
Assume GRH. Let $n_p$ be the least quadratic non-residue modulo a prime $p$. Then 
\begin{align*}
\limsup_{p\to \infty} \frac{n_{p}}{\log^2 p} \leq  \mc{C}(1)^{-2} < 0.7615 \, .
\end{align*} 
\end{theorem}
This result is a particular case of our more general Theorem \ref{Thm3}, below. We also note that the bound in Theorem \ref{Thm2} shows that $\mc{C}(1)^{-2} > 0.7596$, indicating that we have essentially arrived at the limit of this method. In other words, assuming GRH, we prove that $n_p < \frac{380}{499} \log^2 p$ when $p$ is sufficiently large, but that replacing $\frac{380}{499}$ by $\frac{379}{499}$ is not possible using our method. 

\smallskip

It is not difficult to see that the least quadratic non-residue $n_p$ must be prime. Given this fact, it is also natural to study the size of the least prime quadratic residue modulo $p$, which we denote by $r_p$.  Analogous to his conjecture for  $n_p$, I.~M.~Vinogradov conjectured that $r_p \ll_\varepsilon p^\varepsilon$ for any $\varepsilon>0$. The P\'olya-Vinogradov inequality and Siegel’s theorem for exceptional zeros of Dirichlet $L$-functions imply that $r_p \ll_\varepsilon p^{1/2+\varepsilon}$. Yu.~V.~Linnik and A.~I.~Vinogradov \cite{LV} used Burgess' bounds for character sums and Siegel's theorem to show that  $r_p \ll_\varepsilon p^{1/4+\varepsilon}$. The implied constants in both of these bounds are ineffective. Assuming GRH, Ankeny \cite{An} proved that $r_{p}  \ll \log^2p$ and the order of magnitude of this estimate has never been improved. We modify our proof of Theorem \ref{Thm1_LQNR} to prove the following result. 
\begin{theorem}\label{Thm_LPQR}
Assume GRH. Let $r_p$ be the least prime quadratic residue modulo a prime $p$. Then 
\begin{align*}
\limsup_{p\to \infty} \frac{r_{p}}{\log^2 p} \leq  \mc{C}(1)^{-2} < 0.7615 \, .
\end{align*} 
\end{theorem}

\subsection{The least prime in an arithmetic progression} For $q \in \N$ and $a \in \Z$ with $\gcd(a,q) =1$, let $P(a,q)$ be the least prime in the arithmetic progression $\equiv a\,({\rm mod} \, q)$. When it comes to unconditional bounds, a classical result of Linnik establishes that $P(a,q) \ll q^L$ for some universal constant $L$. Heath-Brown \cite{HB} showed that $L = 5.5$ is an admissible value for Linnik's constant, and this result was later improved to $L = 5.2$ by Xylouris \cite{X}.

\smallskip

Under GRH, the asymptotic estimate $P(a,q) \leq (1 + o(1))(\phi(q) \log q)^2$ was established by Bach and Sorensen \cite{BachS}, and the absolute bound $P(a,q) \leq(\phi(q) \log q)^2$, for $q>3$, was established by Lamzouri, Li and Soundararajan in \cite{LLS}. It was suggested in \cite{LLS} that a modification of their proof, with the use of the Brun-Titchmarsh inequality, would yield an asymptotic bound of the form $P(a,q) \leq (1-\delta  + o(1))(\phi(q) \log q)^2$, for some small $\delta>0$. Our next result finds an explicit constant for this inequality, by relating it to the following constant arising from a Fourier optimization problem:\footnote{The motivation for the nomenclature of the constants, $\mc{C}(1)$ in \eqref{20240403_11:21} and $ \mc{C}(3)$ in \eqref{20240403_11:22}, becomes clear in \S \ref{Sec_FO}.}
\begin{align}\label{20240403_11:22}
 \mc{C}(3) := \sup_{0 \neq F \in \mc{A}} \frac{2\pi}{\|F\|_1}\left( \int_{-\infty}^{0} \widehat{F}(t)\,e^{\pi t} \,\d t - \int_{0}^{\infty} \widehat{F}(t)\,e^{\pi t} \,\d t - 2\int_{0}^{\infty} \big|\widehat{F}(t)\big|\,e^{\pi t} \,\d t \right),
\end{align}
where the supremum is taken over the class of functions $\mc{A} = \{F:\R \to \C\,; \ F \in L^1(\R) \,;\ \widehat{F} \ {\rm is\ real \mbox{-}valued}\}$.

\begin{theorem}\label{Thm4}
Assume GRH. Let $q \in \N$ and $a \in \Z$ with $\gcd(a,q) =1$. The least prime $P(a,q)$ that is congruent to $a$ modulo $q$ verifies
\begin{align}\label{20230303_16:13}
\limsup_{q \to \infty} \frac{P(a,q)}{ (\phi(q) \log q)^2} \leq \mc{C}(3)^{-2}  < \frac{8}{9}.
\end{align}
\end{theorem}
We note in Theorem \ref{Thm2} that $\mc{C}(3)^{-2} > 0.8859$, which shows, as in the case of Theorem \ref{Thm1_LQNR}, that we are essentially at the limit of the method.

\subsection{ Fourier optimization} \label{Sec_FO}A Fourier optimization problem can generally be described as a problem in which one prescribes certain constraints for a function and its Fourier transform, and wants to optimize a certain quantity of interest. Some problems of this type turn out to be connected to problems in number theory and other fields, and establishing such bridges is usually an interesting part of the process. Recent examples of Fourier optimization problems arise in connection to: bounds for the modulus and argument of the Riemann zeta-function on the critical line \cite{CCM, CChi, CS}; bounds related to Montgomery's pair correlation conjecture \cite{CCCM, CCLM, CMPR, Ga, QE} and small gaps between zeros of zeta and $L$-functions \cite{BGMM,CGL}; the maximum size of gaps between primes \cite{CMS, CPL, CQE}; the proportion of simple zeros of zeta and $L$-functions \cite{CGL, Mont, Sono}; non-vanishing of $L$-functions at the central point and at low-lying heights \cite{CChiM, FM, ILS}; and the recent breakthroughs in the sphere packing problem and energy-minimizing point configurations \cite{CE, CKMRV, CKMRV2, Go, Vi}.

\smallskip

Throughout the paper, for a real-valued function $G$, we let 
$$G_+(x) := \max\{G(x),0\} \ \  \ {\rm and}  \ \ \ G_-(x) := \max\{-G(x),0\}.$$ 
Then, one plainly has $G(x) = G_+(x) - G_-(x)$ and $|G(x)| =  G_+(x) + G_-(x)$. We consider the following family of optimization problems parametrized by a real number $0 \leq A < \infty$ or $A = \infty$. 

\medskip

\noindent{\it Extremal Problem 1 $(EP1)$.} Given $0 \leq A < \infty$, find
\begin{align}\label{20230302_10:22}
 \mc{C}(A) := \sup_{0 \neq F \in \mc{A}} \frac{2\pi}{\|F\|_1}\left( \int_{-\infty}^{0} \widehat{F}(t)\,e^{\pi t} \,\d t - \int_{0}^{\infty} \widehat{F}_-(t)\,e^{\pi t} \,\d t - A\int_{0}^{\infty} \widehat{F}_+(t)\,e^{\pi t} \,\d t\right),
\end{align}
where the supremum is taken over the class of functions $\mc{A} = \{F:\R \to \C\,; \ F \in L^1(\R) \,;\ \widehat{F} \ {\rm is\ real \mbox{-}valued}\}$. For $A = \infty$, find
\begin{align} \label{20230405_14:20}
 \mc{C}(\infty) := \sup_{0 \neq F \in \mc{A}_{\infty}} \frac{2\pi}{\|F\|_1}\left( \int_{-\infty}^{0} \widehat{F}(t)\,e^{\pi t} \,\d t - \int_{0}^{\infty} \widehat{F}_-(t)\,e^{\pi t} \,\d t \right),
\end{align}
where the supremum is taken over the subclass $\mc{A}_{\infty} \subset \mc{A}$ given by $\mc{A}_{\infty} = \{F:\R \to \C\,; \ F \in L^1(\R) \,; \ \widehat{F} \ {\rm is\ real \mbox{-}valued}   \, ;   \ \widehat{F}(t) \leq 0 \ {\rm for} \ t \geq 0\}$.

\smallskip

As we shall see, the problem of estimating the least character non-residue, generalizing the least quadratic non-residue presented in Theorem \ref{Thm1_LQNR}, is connected to (EP1) in the cases $A = 1/(\ell-1)$ for $\ell \in \N$, $\ell \geq 2$, and hence we pay a closer attention to this problem in the range $0\leq A\leq 1$. On the other hand, the problem of estimating the least prime in an arithmetic progression is connected to (EP1) in the case $A =3$, and we  also pay attention to this case. Our next result collects some qualitative and quantitative properties of $\mc{C}(A)$.

\begin{theorem}\label{Thm1} The following propositions hold:
\begin{itemize}  
 \item[(i)] The supremum in \eqref{20230302_10:22}--\eqref{20230405_14:20} restricted to the functions $F \in \mc{A}$ $($or $F \in \mc{A}_{\infty}$ in case $A = \infty$$)$ with $\widehat{F} \in C^{\infty}_c(\R)$, still yields $ \mc{C}(A)$.
\item[(ii)] The function $A \mapsto  \mc{C}(A)$ is continuous and non-increasing for $0 \leq A \leq \infty$.
\item[(iii)] One has the exact endpoint values ${\mc C}(0) = 2$ and  ${\mc C}(\infty) = 1$.
\item[(iv)] For $0 < A < 1$, one has the bound:
\begin{equation}\label{20230412_13:32}
{\mc C}(A) \geq \max\left\{ 2 \, - \, \frac{2 \, (A\!+\!1) \log\left(\frac{3-A}{A+1}\right)}{|\log A|} + 2A \ \, ; \ \, 1\right\}.
\end{equation}
\end{itemize}
\end{theorem}
\noindent{\sc Remarks.} (i) Noting that the maximum of $x \mapsto 2\,(x+1) \log\!\left(\frac{3-x}{x+1}\right)$ for $0 \leq x < 1$ occurs at $x = 0$, from \eqref {20230412_13:32} one arrives at the weaker, yet slightly simpler, lower bound 
\begin{equation}\label{20230412_13:33}
{\mc C}(A) \geq  2  \, -\frac{2 \log3}{|\log A|}\,,
\end{equation}
which is asymptotically equivalent to \eqref{20230412_13:32} as $A \to 0$. 

\smallskip

(ii) From a number theoretic viewpoint, the constant $A$ in our extremal problem (EP1) arises from (over)estimating the contributions of the primes beyond the least quadratic non-residue or the least prime in an arithmetic progression. In the case of the least quadratic non-residue, we arrive at $A=1$ using the Prime Number Theorem. In the case of the least prime in an arithmetic progression, we arrive at $A=3$ using the Brun-Titchmarsh Theorem. The case $A=\infty$ corresponds to the situation where no analogue of the Brun-Titchmarsh Theorem exists. Such a situation arises when one attempts to bound the least prime in the Chebotarev density theorem in a general setting. In particular, our result that ${\mc C}(\infty) = 1$ can be used to show that the results of Bach and Sorensen \cite[Theorem 3.1]{BachS} concerning the least prime in the Chebotarev density theorem assuming the Extended Riemann Hypothesis are best possible using their method. 

\smallskip

(iii) The Fourier analysis framework developed in this paper should be applicable to proving conditional estimates to a variety of arithmetic problems which rely on input from low-lying zeros of $L$-functions (once one has determined the appropriate constant $A$ for the problem). 

\smallskip

(iv) We pay closer attention to a few specific values of the parameter $A$ (see Theorem \ref{Thm2} below) that are more relevant for our purposes. Nevertheless we decided to include the generic relatively simple lower bound \eqref{20230412_13:32} for completeness (we note that this already refines a related generic bound in  \cite[Theorem 1.3]{LLS}). Problems like finding sharper generic lower bounds, or the exact rate of convergence of $C(A)$ to $2$ as $A \to 0^+$ are certainly interesting but fall outside of the scope of our study.

\smallskip

Although our extremal problem (EP1) can be stated in accessible terms, one quickly realizes that the task of actually finding the exact value of the sharp constant $\mc{C}(A)$ is rather subtle in general. Our next result presents high precision upper and lower bounds for $\mc{C}(A)$ in the most interesting cases for our applications. 

\begin{theorem} \label{Thm2}One has the following bounds:
\begin{align*}
1.31706 &< \mc{C}(1/4) < 1.33509;\\
1.27722 &< \mc{C}(1/3) < 1.28781;\\
1.22112 &< \mc{C}(1/2) < 1.23080;\\
1.14600 & < \mc{C}(1) < 1.14731; \\
1.06082 &< \mc{C}(3) < 1.06240.\\
\end{align*}
\end{theorem}

Establishing upper bounds for $\mc{C}(A)$ is a somewhat delicate issue. Our strategy unveils a connection with a suitable auxiliary extremal problem. Once this bridge is in place, the work to establish upper and lower bounds in Theorem \ref{Thm2} is reduced to finding near-extremal test functions for both the original and the auxiliary extremal problems, a task that is performed via suitable robust programming routines. In principle, our method of proof of Theorem \ref{Thm2} could be adapted to treat, with high precision, any other particular fixed value of $A$.

\subsection{The least character non-residue} Proceeding slightly more generally than in \S \ref{LQNR_Sec_11}, let $q \in \N$ and let $\chi$ be a non-principal Dirichlet character modulo $q$. Set 
\begin{align}\label{20230322_17:52}
n_{\chi} := \min\{n \in \N;\ \chi(n) \neq 0,1\}\,,
\end{align}
known as the {\it least character non-residue}. Under GRH, it was observed by Montgomery \cite[Chapter 13]{Mont_thesis} that Ankeny's methods in \cite{An} can be adapted to show that $n_{\chi}  \ll \log^2q$. The absolute bound $n_{\chi} \leq \log^2q$, if $q$ is not divisible by any prime below $\log^2q$, was established by Lamzouri, Li, and Soundararajan in \cite{LLS}. When it comes to asymptotic results, the uniform estimate $n_{\chi} \leq \big(1 + o(1)\big)\log^2q$ was proved by Bach \cite{Bach}, and later refined in \cite{LLS} to $n_{\chi} \leq \big(0.794 + o(1)\big)\log^2q$. It was observed in \cite{LLS} that one can do better if one knows, a priori, the order of the character $\chi$. For instance, if $\chi$ is cubic then $n_{\chi} \leq \big(0.7 + o(1)\big)\log^2q$, and if $\chi$ is quartic then $n_{\chi} \leq \big(0.66 + o(1)\big)\log^2q$. Our next result provides a sharpening of these inequalities, by relating them to the solution of the extremal problem (EP1). 
\begin{theorem}\label{Thm3}
Assume GRH. Let $q \in \N$ and let $\chi$ be a non-principal Dirichlet character modulo $q$ of order $\ell$. Then the least character non-residue $n_{\chi}$ satisfies
\begin{align*}
\limsup_{q\to \infty} \frac{n_{\chi}}{\log^2 q} \leq  \mc{C} \!\left( \frac{1}{\ell\!-\!1}\right)^{-2}. 
\end{align*} 
\end{theorem}
From Theorem \ref{Thm2} and the lower bound \eqref{20230412_13:32} we immediately have the following corollary.
\begin{corollary} \label{Cor4}Under the same hypotheses of Theorem \ref{Thm3}, we have that as $q \to \infty$:
\begin{itemize}
\item[(i)] If $\ell =2$, then $n_{\chi} \leq \big(0.7615 + o(1)\big) \log^2 q$.
\item[(ii)] If $\ell =3$, then $n_{\chi} \leq \big(0.6707 + o(1)\big) \log^2 q$.
\item[(iii)] If $\ell =4$, then $n_{\chi} \leq \big(0.6131 + o(1)\big) \log^2 q$.
\item[(iv)] If $\ell =5$, then $n_{\chi} \leq \big(0.5765 + o(1)\big) \log^2 q$.
\item[(v)] For $\ell \geq 6$, one has
\begin{align}\label{20250725_12:31}
 n_{\chi} \leq  \left\{  \frac{1}{4} \left(1 - \frac{\frac{\ell}{(\ell-1)} \log\left(\frac{3\ell -4}{\ell}\right)}{\log (\ell\!-\!1)} + \frac{1}{(\ell -1)} \right)^{-2} +o_\ell(1)\right\}\log^2 q.
 \end{align}
\end{itemize}
\end{corollary}

\noindent{\sc Remarks}. (i) Theorem \ref{Thm2} shows that the upper bounds in Corollary \ref{Cor4} (i)\! -\! (iv) cannot be improved beyond $0.7596\ldots$, $0.6601\ldots$, $0.6029\ldots$, and $0.5610\ldots$, respectively, with this method.

\smallskip

(ii) If one uses the simpler bound \eqref{20230412_13:33} instead of \eqref{20230412_13:32} in part (v), one would arrive at the weaker, yet slightly simpler, bound for $\ell \geq 6$,
\begin{align}\label{20230405_16:05}
n_{\chi} \leq  \left\{  \frac{1}{4} \left(1 - \frac{\log 3}{\log (\ell\!-\!1)}\right)^{-2} +o_\ell(1)\right\}\log^2 q.
\end{align}

\smallskip

(iii) We note that Theorem \ref{Thm1_LQNR} is an immediate consequence of Theorem \ref{Thm3} and Corollary \ref{Cor4} upon letting $q$ be prime and letting $\chi$ be the Legendre symbol modulo $q$ (so that $\ell=2$). 

\smallskip

Our strategy to prove Theorem \ref{Thm3} is inspired in the asymptotic bounds of the previous work of Lamzouri, Li, and Soundararajan in \cite{LLS}, with certain differences. Both here and in \cite{LLS}, the idea is to consider a character sum with values $1$ up to a certain point and then $-1$ from that point on (after a suitable cancellation provided by the order of the character). In \cite{LLS}, this is done via the Mellin transform and appropriate contour shifting, while here we alternatively work directly with the Fourier transform and explicit formulas. If one carefully translate the arguments of \cite[Section 6]{LLS} to the Fourier transform language via the appropriate changes of variables, one sees that the extremal problem setup in \cite[Proposition 6.1]{LLS} reduces  to our (EP1), in the case $A = 1/(\ell -1)$, after certain simplifications due to additional restrictions on the test functions. Precisely, in \cite[Proposition 6.1]{LLS} one assumes that $F$ is analytic in the strip $-\frac12 - \varepsilon < {\rm Im}(s) < \frac12 + \varepsilon$, that $|F(s)| \ll (1 + |s|^2)^{-1}$ in this strip as $|{\rm Re} (s)| \to \infty$, and that $\widehat{F} \geq 0$ on $\R$. The first two assumptions are harmless since we prove in Theorem \ref{Thm1} (i) that $F$ can be taken entire and Schwartz on $\R$ for (EP1). However, the additional assumption that $\widehat{F} \geq 0$ on $\R$ makes the two problems different. It is not clear that near-extremizers of problem (EP1) will have this property (most likely not; see the discussion in Section \ref{Sec5_Computational}). Some of our test functions used in Theorem \ref{Thm2} already show a small oscillation of the sign of $\widehat{F}$, e.g. the test function for the lower bound when $A = 1$.

\smallskip

Here we not only estimate the limit of the method in the particular situations of low order characters (Theorem \ref{Thm2}) but also in the regime where the order $\ell$ is large. This is given by \eqref{20250725_12:31}\! -\! \eqref{20230405_16:05}, in which one gets the constant $1/4$ as the limit. The same limiting constant $1/4$ was also obtained in \cite[Theorem 1.3]{LLS} with a slightly larger upper bound.

\subsection{Structure of the paper} The rest of the paper, in which we prove the main results stated in this introduction, can be broadly divided into three independent parts:

\begin{itemize}
\item{\it The analysis part}. This is Section \ref{Sec2_Approx}, where we prove Theorem \ref{Thm1}.
\item {\it The number theory part}. This is comprised of Sections \ref{secAuxLem}, \ref{Sec3_CNR}, and \ref{Section_primes_AP}, where we prove Theorems \ref{Thm3}, \ref{Thm_LPQR}, and \ref{Thm4}.
\item {\it The computational part}. This is Section \ref{Sec5_Computational}, where we discuss the computer-assisted proof of Theorem \ref{Thm2}.
\end{itemize}

Finally, Appendix A at the end of the paper discusses the connection between bounds for the least character non-residue $n_\chi$ and conjectural bounds for $S(t,\chi) = \frac{1}{\pi} \arg L(\tfrac{1}{2}+it,\chi)$.

\section{Extremal problems: proof of Theorem \ref{Thm1}} \label{Sec2_Approx}

Let us introduce another important tool in our conceptual framework, that shall be relevant not only for the proof of Theorem \ref{Thm1} but also for the proof of Theorem \ref{Thm2} in the upcoming Section \ref{Sec5_Computational}.

\subsection{An auxiliary extremal problem} \label{SubSec_Dual} Let ${\bf 1}_{\R_-}$ be the characteristic function of $(-\infty, 0)$ and ${\bf 1}_{\R_+}$ be the characteristic function of $(0, \infty)$. We consider the following auxiliary problem. 

\medskip

\noindent{\it Extremal Problem 2 $(EP2)$.} Given $0 \leq A \leq \infty$, find
\begin{align*}
\mc{C}^*(A) = \inf_{0 \neq \psi \in \mc{B}_A}\ 2\pi \, \left\|\,\widehat{{\bf 1}_{\R_-}\, e^{\pi (\cdot)}} - \widehat{{\bf 1}_{\R_+} \, \psi}\, \right\|_{\infty}\,,
\end{align*}
where the infimum is taken over the class $\mc{B}_A = \{\psi:\R_+ \to \R;\ \psi \in L^1(\R_+) \ {\rm and} \  -e^{\pi t} \leq \psi(t) \leq A\,e^{\pi t} \ {\rm for \ all} \ t \in \R_+\}$.

\smallskip

\noindent{\sc Remark:} If $A=\infty$, the last condition is understood simply as $-e^{\pi t} \leq \psi(t)$ for all $ t \in \R_+$.

\smallskip

It is clear from the definition of the extremal problem (EP2) that the map $A \mapsto \mc{C}^*(A)$ is non-increasing. We prove the following relation between our extremal problems (EP1) and (EP2). 

\begin{proposition}\label{Prop6}
For each $0 \leq A \leq \infty$ one has
\begin{equation}\label{20230307_16:58}
\mc{C}(A) \leq \mc{C}^*(A). 
\end{equation}
\end{proposition}

\begin{proof} Let $0 \neq F \in \mc{A}$ (or $0 \neq F \in \mc{A}_{\infty}$ in case $A = \infty$) be given. Using the multiplication formula for the Fourier transform, for any $\psi \in \mc{B}_A$ we have
\begin{align*}
\int_{-\infty}^{0} \widehat{F}(t)\,e^{\pi t} \,\d t  & - \int_{0}^{\infty} \widehat{F}_-(t)\,e^{\pi t} \,\d t - A\int_{0}^{\infty} \widehat{F}_+(t)\,e^{\pi t} \,\d t 
 \leq  \int_{-\infty}^{\infty} \widehat{F}(t) \Big( {\bf 1}_{\R_-}(t)\, e^{\pi t} - {\bf 1}_{\R_+} (t)\, \psi(t)\Big) \,\d t\\
& = \int_{-\infty}^{\infty} F(x) \Big( \widehat{{\bf 1}_{\R_-}\, e^{\pi (\cdot)}} - \widehat{{\bf 1}_{\R_+} \, \psi}\Big)(x) \,\d x \leq \|F\|_1 \, \left\|\,\widehat{{\bf 1}_{\R_-}\, e^{\pi (\cdot)}} - \widehat{{\bf 1}_{\R_+} \, \psi}\, \right\|_{\infty}.
\end{align*}
This plainly leads us to \eqref{20230307_16:58}.
\end{proof}

\subsection{Approximations: proof of Theorem \ref{Thm1} (i)} Our argument below works in both cases $0\leq A < \infty$ and $A = \infty$. Note first that $\mc{C}(A) >0$; to see this just take any test function $F \in \mc{A}$ with ${\rm supp}\big(\widehat{F}\big) \subset (-\infty, 0]$. Start with any $0 \neq F \in \mc{A}$ (resp. $0 \neq F \in \mc{A}_{\infty}$ if $A = \infty$) and define
\begin{equation*}
J_A(F):= \frac{2\pi}{\|F\|_1}\left( \int_{-\infty}^{0} \widehat{F}(t)\,e^{\pi t} \,\d t - \int_{0}^{\infty} \widehat{F}_-(t)\,e^{\pi t} \,\d t - A\int_{0}^{\infty} \widehat{F}_+(t)\,e^{\pi t} \,\d t\right)
\end{equation*}
(resp.~with the last integral on the right removed if $A = \infty$). We may assume that $J_A(F) >0$. Let $K(x) = \big(\sin(\pi x)/(\pi x) \big)^2$ be the Fej\'{e}r kernel and recall that $\widehat{K}(t) = (1 - |t|)_+$. For $\varepsilon >0$, define $K_\varepsilon(x) := \varepsilon^{-1}K(x/\varepsilon)$. Since $\widehat{F *K_{\varepsilon}}(t) = \widehat{F}(t)\widehat{K}(\varepsilon t)$, an application of dominated convergence in the numerator, together with the fact that $\|F *K_{\varepsilon}\|_1 \to  \|F\|_1$ as $\varepsilon \to 0$, yields $\lim_{\varepsilon \to 0} J_A(F *K_{\varepsilon}) = J_A(F)$. We may hence assume that our test function $F$ is bandlimited. 

\smallskip

Let $\eta \in C^{\infty}_c(\R)$ be an even and non-negative function with $\supp(\eta) \subset [-1,1]$ and $\int_{-1}^1 \eta(x)\,\d x =1$. Again, for $\varepsilon >0$, let $\eta_\varepsilon(x) := \varepsilon^{-1}\eta(x/\varepsilon)$. Assume that $\supp(\widehat{F}) \subset[-\Lambda, \Lambda]$ and define $F^{\varepsilon}(x) := F(x)\,\widehat{\eta}(\varepsilon x)\,e^{-2 \pi i x \varepsilon}$. Then $\widehat{F^{\varepsilon}}(t) = \widehat{F}*\eta_{\varepsilon}(t + \varepsilon) \in C^{\infty}_c(\R)$ and $\supp(\widehat{F^{\varepsilon}}) \subset [-\Lambda - 2\varepsilon, \Lambda]$ (this small translation by $\varepsilon$ is just to guarantee that $F^{\varepsilon} \in \mc{A}_{\infty}$ in case $F \in \mc{A}_{\infty}$). Note that $\|F^{\varepsilon}\|_1 \to \|F\|_1$ as $\varepsilon \to 0$, by dominated convergence. Since $\widehat{F}$ is uniformly continuous, note that $\widehat{F^{\varepsilon}}  \to \widehat{F}$ uniformly as $\varepsilon \to 0$, and hence $\widehat{F^{\varepsilon}}_-  \to \widehat{F}_-$ and $\widehat{F^{\varepsilon}}_+  \to \widehat{F}_+$ uniformly as well. Yet another application of dominated convergence in the numerator yields $\lim_{\varepsilon \to 0} J_A(F^{\varepsilon}) = J_A(F)$. This concludes the proof of this part.

\subsection{Continuity: proof of Theorem \ref{Thm1} (ii)} From the definition of the problem, it is clear that the function $A \mapsto \mc{C}(A)$ is non-increasing. Let us show that it is continuous up to $\infty$ (i.e. we also want to show that $\lim_{A \to \infty} \mc{C}(A) = \mc{C}(\infty)$).

\smallskip

For a fixed $0 \leq A < \infty$, given $\varepsilon >0$, let $F \in \mc{A}$ be such that $J_A(F) \geq \mc{C}(A) - \varepsilon$. Then, by the continuity of the numerator, there exists $\delta = \delta(\varepsilon, F) >0$ such that if $|B - A| \leq \delta$ then  $J_B(F) \geq J_A(F) - \varepsilon$. Hence $\mc{C}(B)\geq J_B(F) \geq J_A(F) - \varepsilon \geq \mc{C}(A) - 2\varepsilon$. Since $\varepsilon >0$ was arbitrary, we get
\begin{equation}\label{20230411_17:10} 
\liminf_{B \to A} \mc{C}(B) \geq \mc{C}(A).
\end{equation}
Since $A \mapsto \mc{C}(A)$ is non-increasing, note that \eqref{20230411_17:10} already implies that this map is continuous at $A = 0$. Note also that \eqref{20230411_17:10} is trivially true for $A = \infty$.

\smallskip

Now assume that $0 < A < \infty$ and let $\{A_n\}_{n \geq 1}$ be a sequence such that $A_n > 0$ for each $n$ and $A_n \to A$ as $n \to \infty$. Let $F_n \in \mc{A}$ be normalized such that $\|F_n\|_1 = 1$ and $J_{A_n}(F_n) \geq \mc{C}(A_n) - \frac{1}{n}$. Since $\|\widehat{F_n}\|_{\infty} \leq \|F_n\|_1 = 1$, 
\begin{align*}
 \left|\int_{-\infty}^{0} \widehat{F_n}(t)\,e^{\pi t} \,\d t \right| \leq \int_{-\infty}^{0} e^{\pi t} \,\d t  = \frac{1}{\pi},
\end{align*}
and from the fact that $\mc{C}(A_n) >0$ we obtain 
\begin{align*}
A_n \int_{0}^{\infty} \big(\widehat{F_n}\big)_+(t)\,e^{\pi t} \,\d t  \leq \frac{1}{\pi}.
\end{align*}
Therefore,
\begin{align*}
\big|J_A(F_n) - J_{A_n}(F_n)\big| = 2\pi \left| (A - A_n) \int_{0}^{\infty} \big(\widehat{F_n}\big)_+(t)\,e^{\pi t} \,\d t\right| \leq \frac{2\,|A-A_n|}{A_n}\,,
\end{align*}
and note that this last quantity goes to $0$ as $n \to \infty$. This plainly leads us to  
\begin{align}\label{20230411_17:11} 
\mc{C}(A) \geq \limsup_{n \to \infty} J_A(F_n) = \limsup_{n \to \infty} J_{A_n}(F_n) = \limsup_{n \to \infty} \mc{C}(A_n).
\end{align}
The desired continuity at $0 < A < \infty$ then follows from \eqref{20230411_17:10} and \eqref{20230411_17:11}. 

\smallskip

The argument in the previous paragraph does not immediately work to prove the continuity at $A = \infty$ since the functions $F_n$ may not be in the class $\mc{A}_{\infty}$. We have to be a bit more careful in this case, and argue using the auxiliary extremal problem (EP2). We show that $\mc{C}^*(A) = 1 + o(1)$ as $A \to \infty$. Recall that 
\begin{align*}
2\pi\, \widehat{{\bf 1}_{\R_-}\, e^{\pi (\cdot)}}(x) = \frac{2}{1 - 2i x}.
\end{align*}
The intuitive idea is the following: if we consider $\psi = \frac{1}{2\pi} {\bm \delta}$, where ${\bm \delta}$ is the Dirac delta at the origin, we would have 
\begin{align*}
2\pi\, \widehat{{\bf 1}_{\R_-}\, e^{\pi (\cdot)}}(x) - 2\pi \, \widehat{\psi}(x) = \frac{2}{1 - 2i x} - 1= \frac{1 + 2ix}{1 - 2i x}\,,
\end{align*}
and hence $2\pi\big\|\,\widehat{{\bf 1}_{\R_-}\, e^{\pi (\cdot)}} - \widehat{ \psi}\, \big\|_{\infty} = 1$. With this target example in mind, we argue with a suitable admissible approximation. For $A$ large (and finite), let us choose a test function $\psi_A \in \mc{B}_A$ for  (EP2) given by $\psi_A(t) = A\,{\bf 1}_{[0,1/(2\pi A)]}(t)$. Then, note that 
\begin{equation*}
2\pi \, \widehat{{\bf 1}_{\R_+} \, \psi_A}(x) = \frac{1 - e^{- i x/A}}{ i x/A}.
\end{equation*}
By the mean value inequality, $|1 - e^{i \theta}| \leq |\theta|$ for $\theta \in \R$, and we get
\begin{align}\label{20230420_19:16}
\big\|2\pi \, \widehat{{\bf 1}_{\R_+} \, \psi_A}(x)\big\|_{\infty} \leq 1.
\end{align} 
Also, note that $\lim_{x \to 0} 2\pi \, \widehat{{\bf 1}_{\R_+} \, \psi_A}(x) = 1$. 

\smallskip

Now let $\varepsilon >0$ be given, and let $N = N(\varepsilon)$ be large so that 
\begin{equation}\label{20230420_19:15}
\left| \frac{2}{1 - 2i x}\right| \leq \varepsilon
\end{equation}
if $|x| \geq N$. Fix $A_0 = A_0(\varepsilon, N)$, so that if $A \geq A_0$ we have 
\begin{equation}\label{20230420_15:43}
\left|\frac{1 - e^{- i x/A}}{ i x/A}-1 \right| \leq \varepsilon
\end{equation}
if $|x| \leq N$. Hence, if $A \geq A_0$ and $|x| \leq N$, we use the triangle inequality and \eqref{20230420_15:43} to get 
\begin{align}\label{20230420_19:19}
\left|2\pi\, \widehat{{\bf 1}_{\R_-}\, e^{\pi (\cdot)}}(x) -2\pi \, \widehat{{\bf 1}_{\R_+} \, \psi_A}(x) \right| & \leq \left| \frac{2}{1 - 2i x} - 1\right| + \left|1 - \frac{1 - e^{- i x/A}}{ i x/A}\right|  \leq 1 + \varepsilon.
\end{align}
On the other hand, if $|x| > N$, we use the triangle inequality, with \eqref{20230420_19:16} and \eqref{20230420_19:15}, to get
\begin{align}\label{20230420_19:20}
\left|2\pi\, \widehat{{\bf 1}_{\R_-}\, e^{\pi (\cdot)}}(x) -2\pi \, \widehat{{\bf 1}_{\R_+} \, \psi_A}(x) \right| & \leq \left| \frac{2}{1 - 2i x} \right| + \big\|2\pi \, \widehat{{\bf 1}_{\R_+} \, \psi_A}(x)\big\|_{\infty}  \leq 1 + \varepsilon.
\end{align} 
Inequalities \eqref{20230420_19:19} and \eqref{20230420_19:20} plainly lead us to (recall that $A \mapsto \mc{C}^*(A)$ is non-increasing)
\begin{align*}
\lim_{A \to \infty} \, \mc{C}^*(A) \leq 1 + \varepsilon.
\end{align*}
Since $\varepsilon>0$ is arbitrary, and in light of Proposition \ref{Prop6} and the fact that the map $A \mapsto \mc{C}(A)$ is non-increasing, we conclude that 
\begin{align}\label{20230420_21:32}
\mc{C}(\infty) \leq \lim_{A \to \infty} \, \mc{C}(A) \leq \lim_{A \to \infty} \, \mc{C}^*(A) \leq 1.
\end{align}
Once we prove that $\mc{C}(\infty) = 1$, which we shall do in the next subsection, the continuity at $A = \infty$ plainly follows from \eqref{20230420_21:32} (and so does the fact that $\mc{C}^*(\infty) = 1$).

\subsection{Endpoint values: proof of Theorem \ref{Thm1} (iii)} 

\subsubsection{The value of $\mc{C}(0)$} \label{subsub_241}If $A = 0$ and $F \in \mc{A}$, using the fact that $\|\widehat{F}\|_{\infty} \leq \|F\|_1$ note that 
\begin{align*}
J_{0}(F) \leq \frac{2\pi}{\|F\|_1} \int_{-\infty}^{0} \widehat{F}(t)\,e^{\pi t} \,\d t \leq 2\pi\int_{-\infty}^{0} e^{\pi t} \,\d t = 2. 
\end{align*}
This implies that $\mc{C}(0) \leq 2$. To show that we indeed have equality, consider again the Fej\'{e}r kernel $K(x) = \big(\sin(\pi x)/(\pi x) \big)^2$ and recall that $\widehat{K}(t) = (1 - |t|)_+$. For $\varepsilon >0$, define $K_\varepsilon(x) := \varepsilon^{-1}K(x/\varepsilon)$ and hence $\widehat{K_\varepsilon}(t) = \widehat{K}(\varepsilon t)$. Then 
\begin{align*}
J_{0}(K_\varepsilon) = 2\pi \int_{-\infty}^{0} (1 - |\varepsilon t|)_+\,e^{\pi t} \,\d t \to 2\pi \int_{-\infty}^{0} \,e^{\pi t} \,\d t = 2
\end{align*}
as $\varepsilon \to 0$, by dominated convergence. This shows that $\mc{C}(0) = 2$.

\smallskip

\noindent{\sc Remark:} Since $2 = \mc{C}(0) \leq \mc{C}^*(0)$, the test function $\psi =0$ in the problem (EP2) yields $\mc{C}^*(0) = 2$.

\subsubsection{The value of $\mc{C}(\infty)$}\label{sec:c-infty} Consider the function $F \in \mc{A}_{\infty}$ given by 
\begin{align*}
F(x) = \frac{2}{\pi(1 + 2ix)^2}.
\end{align*}
Note that $\|F\|_1 =1$ and $\widehat{F}(t) = -2\pi t \,e^{\pi t}\,{\bf 1}_{\R_-}(t)$. Then 
\begin{align}\label{20230420_22:17}
{\mc C}(\infty) \geq J_{\infty}(F) = 2\pi \int_{-\infty}^0 (-2\pi t) \,e^{2\pi t}\,\d t = 1.
\end{align}
From \eqref{20230420_21:32} and \eqref{20230420_22:17} we conclude that ${\mc C}(\infty) = 1$.

\subsection{Lower bound near $A=0$: proof of Theorem \ref{Thm1} (iv)} Fix $0< A <1$ and we keep the notation for $K_\varepsilon$ as in \S\ref{subsub_241}. The idea is to consider a translated function of the form $\widehat{F}(t) = \widehat{K_\varepsilon}(t+c)$, for suitable $\varepsilon >0$ and $c \geq 0$ that depend on $A$. Note that for such $F$ one has $\|F\|_1 = \|K_{\varepsilon}\|_1 = 1$. Assuming that $0 \leq c \leq \frac{1}{\varepsilon}$ one proceeds with the explicit computation: 
\begin{align}\label{20230412_11:08}
\begin{split}
J_{A}(F) & = 2\pi \left( \int_{-\infty}^{0} \widehat{K_\varepsilon}(t+c)\,e^{\pi t} \,\d t - A \int_{0}^{\infty} \widehat{K_\varepsilon}(t+c)\,e^{\pi t} \,\d t \right)\\
& = 2\pi e^{-\pi c}\left( \int_{-1/\varepsilon}^{c} \widehat{K_\varepsilon}(y)\,e^{\pi y} \,\d y - A \int_{c}^{1/\varepsilon} \widehat{K_\varepsilon}(y)\,e^{\pi y} \,\d y \right)\\
& = 2 - \frac{4\varepsilon e^{-\pi c}}{\pi} - \frac{2 \varepsilon (\pi c - 1)}{\pi} + \frac{2 \varepsilon e^{-\pi c - \frac{\pi}{\varepsilon}}}{\pi} + 2A \left(1 - \frac{\varepsilon e^{-\pi c + \frac{\pi}{\varepsilon}}}{\pi} - \frac{\varepsilon(\pi c - 1)}{\pi}\right).\\
\end{split} 
\end{align}
We now take $\varepsilon = \pi / \log\big(1/A\big)$, and the last line of \eqref{20230412_11:08} becomes 
\begin{align}\label{20230412_11:21}
=  2 - \frac{2}{\log \left(\frac{1}{A}\right)}\left( (3-A)e^{-\pi c} + (A+1) (\pi c -1)\right) + 2A.
\end{align}
The maximum of \eqref{20230412_11:21} over $c>0$ occurs at $c = \frac{1}{\pi} \log\big((3-A)/(A+1)\big)$ (note that $0\leq c \leq \frac{1}{\varepsilon}$ in this case). With these particular choices, \eqref{20230412_11:08} and \eqref{20230412_11:21} yield the desired lower bound
\begin{align*}
\mc{C}(A) \geq J_{A}(F) = 2 -  \frac{2(A+1) \log\left(\frac{3-A}{A+1}\right)}{\log \left(\frac{1}{A}\right)} + 2A.
\end{align*}
Naturally, one also has the alternative bound $\mc{C}(A) \geq \mc{C}(\infty) = 1$. This concludes the proof.

\section{Explicit formula and auxiliary lemmas} \label{secAuxLem}
We start by recalling the classical Guinand-Weil explicit formula\footnote{This formula might sometimes be referred to as the Riemann-Weil explicit formula or Weil explicit formula \cite{BRS, MV_Book}. The historical account for the nomenclature used here is explained in \cite[\S 12.3]{MV_Book}.}, in the context of Dirichlet $L$-functions associated to primitive Dirichlet characters. The proof of the next result can be established by modifying the proof of \cite[Theorem 5.12]{IK}; see for instance \cite[Lemma 5]{CarFinder}.

\begin{lemma}[Guinand-Weil explicit formula] \label{GW_lemma} Let $h(s)$ be analytic in the strip $|\mathrm{Im} \, s| \le \tfrac12+\varepsilon$ for some $\varepsilon>0$, and assume that $|h(s)| \ll (1+|s|)^{-(1+\delta)}$ for some $\delta>0$ when $|\mathrm{Re} \, s| \to \infty$. Let $\chi$ be a primitive Dirichlet character modulo $q$. Then
\begin{align}\label{20230323_09:47}
\begin{split}
\sum_{\rho_{\chi}} h\!\left(\frac{\rho_{\chi}-\tfrac12}{i}\right) &= \widehat{h}(0) \frac{\log (q/\pi)}{2\pi}   +  \frac{1}{2\pi}  \int_{-\infty}^{\infty} h(u) \, \mathrm{Re} \,  \frac{\Gamma'}{\Gamma}\left(\frac{2 - \chi(-1)}{4} + \frac{iu}{2}\right)  \mathrm{d}u 
\\
& \qquad \qquad - \frac{1}{2\pi}  \sum_{n\geq2}\frac{\Lambda(n)}{\sqrt{n}} \left\{  \chi(n) \, \widehat{h}\!\left( \frac{\log n}{2\pi} \right) + \overline{\chi(n)}\,\widehat{h}\!\left( -\frac{\log n}{2\pi} \right)  \right\},
\end{split}
\end{align}
where the sum on the left-hand side runs over the non-trivial zeros $\rho_{\chi}$ of $L(s,\chi)$, and $\Lambda(n)$ is the von Mangoldt function defined to be $\log p$ if $n=p^k$, $p$ a prime and $k\ge 1$, and zero otherwise. 
\end{lemma}

Our proofs of Theorems \ref{Thm_LPQR}, \ref{Thm4}, and \ref{Thm3} rely on a suitable application of Lemma \ref{GW_lemma}. We now state a few lemmas which will constitute the building blocks of the proofs of these theorems, to be assembled in the next sections.

\smallskip

\begin{lemma}\label{lemSumOverZeros}
    Assume GRH. Let $\chi$ be a primitive Dirichlet character modulo $q$ and let $h \in C^1(\R)$ be fixed such that $|h(x)|+ |h'(x)| \ll (1+|x|)^{-(1+\delta)}$ for some $\delta>0$ and for all $x \in \R$. Then
    \begin{align}
        \left|\sum_{\rho_{\chi}}  h\!\left(\frac{\rho_{\chi}-\tfrac12}{i}\right) \right| \leq \frac{\log q}{2 \pi} \, \| h\|_1 + O_h\!\left(\frac{\log q}{\log \log q}\right).
    \end{align}
\end{lemma}

\begin{proof}
For $T>0$, let $N(T, \chi)$ be the number of zeros $\rho_{\chi} = \beta_{\chi} + i \gamma_{\chi}$ of $L(s, \chi)$ with $0< \beta_{\chi} <1$ and $0 \leq \gamma_{\chi} \leq T$ (any zeros with $\gamma_{\chi} =0$ or $\gamma_{\chi} = T$ should be counted with multiplicity $\tfrac12$). Letting 
$$S(T,  \chi)= \frac{1}{\pi} \,{\rm arg} \,L\big(\tfrac12 + iT,  \chi\big),$$
where the argument is defined by continuous variation from $+\infty+iT$, we have the unconditional identity \cite[Corollary 14.6]{MV_Book} 
\begin{equation}
\label{20230327_13:02}
N(T,  \chi) = \frac{T}{2\pi} \log\frac{q\, T}{2\pi} - \frac{T}{2\pi} + S(T,  \chi) - S(0,  \chi) - \frac{ \chi(-1)}{8} + O\!\left(\frac{1}{T\!+\!1}\right).
\end{equation}
Note also that if $\rho_{\chi} = \beta_{\chi} + i \gamma_{\chi}$ is a zero of $L(s,  \chi)$ with $-T \leq \gamma \leq 0$, then $\overline{\rho_{\chi}} = \beta_{\chi} - i \gamma_{\chi}$ is a zero of $L(s, \overline{ \chi})$. 

\smallskip

Assuming GRH, Selberg \cite{Selberg} proved that
\begin{align}
\label{20230327_13:04}
|S(T,  \chi)| \ll  \frac{\log \big(q(T\!+\!3)\big)}{\log \log \big(q(T\!+\!3)\big)} \leq \frac{\log q}{\log \log q} + \log (T\!+\!3).
\end{align}
Note that the implicit constants in the error terms of \eqref{20230327_13:02} and \eqref{20230327_13:04} are absolute. Under GRH, we may write the non-trivial zeros as $\rho_\chi = \tfrac12 + i \gamma_\chi$, so that using summation by parts with \eqref{20230327_13:02} and \eqref{20230327_13:04}, we have
\begin{align}
\label{20230412_22:28}
\begin{split}
  \left|\sum_{\rho_{\chi}}  h\!\left(\frac{\rho_{\chi}-\tfrac12}{i}\right) \right| = \left|\sum_{\gamma_{\chi}} h\!\left(\gamma_{\chi}\right) \right| & \leq \sum_{\gamma_{\chi}} \big|h\!\left(\gamma_{\chi}\right)\big|  =\int_{0^-}^{\infty} |h(t)| \, \d N(t, \chi) +  \int_{0^-}^{\infty} |h(-t)| \, \d N(t, \overline{\chi})\\
& =   \frac{\log q}{2\pi} \,  \|h\|_1 + O_h\!\left( \frac{\log q}{\log \log q}\right).
\end{split}
\end{align}    
This proves the lemma.
\end{proof}

Next, we need a way to estimate sums over primes.
\begin{lemma}\label{lemSumOverPrimes}
    Let $g \in C_c^1(\R)$ and $m \in \N$ be such that $\supp \, g \subset [-\frac{\log m}{\pi},\frac{\log m}{\pi} ]$. Then
    \begin{itemize}
        \item [(i)]$\displaystyle
        \frac{1}{2\pi}  \sum_{2 \leq n < m}\frac{\Lambda(n)}{\sqrt{n}}   \, g\!\left( \frac{\log n}{2\pi} \right)  =   \int_0^{\frac{\log m}{2 \pi}} g(t) \, e^{\pi t} \,\d t + O\!\left((\| g\|_{1} + \| g'\|_{1}) \log^2 m\right);$
        \smallskip
        \item [(ii)] $\displaystyle
        \frac{1}{2\pi}  \sum_{n \geq m}\frac{\Lambda(n)}{\sqrt{n}}   \, g \!\left( \frac{\log n}{2\pi} \right)  =   \int_{\frac{\log m}{2 \pi}}^\infty g(t) \, e^{\pi t} \,\d t + O\!\left((\| g\|_{1} + \| g'\|_{1})\log^2 m\right)$.
    \end{itemize}
\end{lemma}

\begin{proof}
Recall that, under RH, we have \cite[Theorem 13.1]{MV_Book}
\begin{align}
\label{20230327_15:44}
\psi(x) :=\sum_{n\leq x} \Lambda(n) = x + O\big(\sqrt{x}\,\log^2x\big).
\end{align} 
Summation by parts yields
\begin{align}
\begin{split}
\sum_{2 \leq n < m}\frac{\Lambda(n)}{\sqrt{n}}   \, g\!\left( \frac{\log n}{2\pi} \right)  & =  \int_2^{m} g\!\left( \frac{\log x}{2\pi} \right) \frac{\d x}{\sqrt{x}} + O\!\left((\| g\|_{1} + \| g'\|_{1})\log^2 m\right) \\
& = 2 \pi \int_0^{\frac{\log m}{2 \pi}} g(t) \, e^{\pi t} \,\d t + O\!\left( (\| g\|_{1} + \| g'\|_{1}) \log^2 m\right),
\end{split}
\end{align}
yielding item (i). The same procedure yields item (ii).
\end{proof}

Often, we will have to work with sums involving non-primitive characters, which are not covered directly by Lemma \ref{GW_lemma}. The following result will be useful to treat this situation.

\begin{lemma}\label{lemNonPrimChar}
Let $\chi$ be a Dirichlet character modulo $q$ and let $\chi^*$ denote the unique
primitive Dirichlet character that induces $\chi$. If $g \in L^\infty(\R)$, then
\begin{align*}
&\left| \sum_{n\geq2} \frac{\Lambda(n)}{\sqrt{n}} \left\{  \left(\chi^*(n) - \chi(n)\right)\, g\!\left( \frac{\log n}{2\pi} \right) + \overline{\big(\chi^*(n) - \chi(n)\big)}\,g\!\left( -\frac{\log n}{2\pi} \right)  \right\}\right|   = O\left( \|g\|_\infty \sqrt{\log q}\right). 
\end{align*}
\end{lemma}
\begin{proof}
    By the definition of $\chi^*$ and the triangle inequality,
    \begin{align*}
        \Bigg| \sum_{n\geq2} \frac{\Lambda(n)}{\sqrt{n}} \bigg\{  \left(\chi^*(n) - \chi(n)\right)\, g\!\left( \frac{\log n}{2\pi} \right) &+ \overline{\big(\chi^*(n) - \chi(n)\big)}\,g\!\left( -\frac{\log n}{2\pi} \right)  \bigg\}\Bigg| \leq 2 \sum_{p|q}\sum_{k\geq1} \frac{\log p}{p^{k/2}}  \, \left|g\!\left( \frac{\log p^k}{2\pi} \right)\right| \\
        &\ll_g  \sum_{p |q} \frac{\log p}{\sqrt{p}} + \sum_{p|q} \frac{\log p}{p}  + \sum_{p|q}\sum_{k\geq 3 }\frac{\log p}{p^{k/2}} \\
        &\leq 2 \sum_{p |q} \frac{\log p}{\sqrt{p}}+  \sum_{p}\sum_{k\geq 3 }\frac{\log p}{p^{k/2}} \\
        &\ll \sqrt{\log q},
    \end{align*}
    where the last estimate comes from the bounds 
    \begin{equation}
    \label{20230325_16:04}
    \sum_p\sum_{k\geq 3 }\frac{\log p}{p^{k/2}} \ll 1 \quad \text{ and } \quad \sum_{p|q} \frac{\log p}{p} \leq \sum_{p |q} \frac{\log p}{\sqrt{p}} \ll \sqrt{\log q}.
    \end{equation}
    For the final estimate in \eqref{20230325_16:04}, we split the sum at height $\log q$ and note that 
    $$\sum_{p\leq \log q} \frac{\log p}{\sqrt{p}} \ll \sqrt{\log q} \ \ \ \ {\rm and} \ \ \ \sum_{\substack{p|q \\ \log q \leq p}} \frac{\log p}{\sqrt{p}} \leq \frac{1}{\sqrt{\log q}} \sum_{p|q} \log p \leq \sqrt{\log q},$$
  completing the proof of the lemma.   
\end{proof}

\section{Character non-residues: proof of Theorem \ref{Thm3}}\label{Sec3_CNR}

\subsection{Setup} Now let $\chi$ be a non-principal Dirichlet character modulo $q$ of order $\ell$, and let $n_{\chi}$ be the least character non-residue as defined in \eqref{20230322_17:52}. We also let $\chi^*$ denote the unique primitive Dirichlet character that induces $\chi$.
Throughout the proof below we set 
$$\Delta:=\log n_{\chi}/(2\pi),$$ 
and hence $n_{\chi} = e^{2\pi \Delta}$. From the work of Lamzouri, Li, and Soundararajan \cite{LLS}, one has $n_{\chi} \leq \big(0.794 + o(1)\big)\log^2q$, and hence we may assume without loss of generality that $q$ is large and that $n_{\chi} \leq \log^2q$.

\smallskip

Let $F \in \mc{A}$ with $\widehat{F} \in C^{\infty}_c(\R)$ be fixed, say with $\supp\big( \widehat{F}\big) \subset [-N, N]$. Let
\begin{equation}\label{20230323_09:59}
\widehat{h}(t) := \frac{\widehat{F}(t - \Delta) + \widehat{F}(-t - \Delta)}{2}
\end{equation}
and assume, without loss of generality, that $\Delta \geq N$. Since $\widehat h \in C_c^\infty (\R)$, the Paley-Wiener theorem \cite[Theorem IX.11]{RS} tells us that $h$ extends to an analytic function of $\C$, which in the strip $|\mathrm{Im} \, s| \le 1$ has the decay  $|h(s)| + |h'(s)|\ll (1+|s|)^{-(1+\delta)}$ for some $\delta>0$ when $|\mathrm{Re} \, s| \to \infty$. For each $j = 1,2,\ldots, \ell -1$ we thus apply the formula \eqref{20230323_09:47} for the primitive character $(\chi^j)^*$ and the test function $h$ given by \eqref{20230323_09:59}. Assuming GRH for Dirichlet $L$-functions, we may write the non-trivial zeros as $\rho_{(\chi^j)^*} = \tfrac12 + i \gamma_{(\chi^j)^*}$, with $\gamma_{(\chi^j)^*} \in \R$. Adding and subtracting the sum over primes with $\chi^j$, and rearranging terms we get (below we let $q_j$ be the modulus of $(\chi^j)^*$)
\begin{align}\label{20230323_10:21}
- & \sum_{\gamma_{(\chi^j)^*}} h\!\left(\gamma_{(\chi^j)^*}\right) + \widehat{h}(0)  \frac{\log (q_j/\pi)}{2\pi}   +  \frac{1}{2\pi}  \int_{-\infty}^{\infty} h(u) \, \mathrm{Re} \,  \frac{\Gamma'}{\Gamma}\left(\frac{2 - (\chi^j)^*(-1)}{4} + \frac{iu}{2}\right)  \mathrm{d}u  \nonumber \\
& \ \ \ \ \  =  
\frac{1}{2\pi}  \sum_{n\geq2} \frac{\Lambda(n)}{\sqrt{n}} \left\{  \chi^j(n) \, \widehat{h}\!\left( \frac{\log n}{2\pi} \right) + \overline{\chi^j(n)}\,\widehat{h}\!\left( -\frac{\log n}{2\pi} \right)  \right\}\\
&  \ \ \ \ \  \ \ \ \ \ + \frac{1}{2\pi}  \sum_{n\geq2} \frac{\Lambda(n)}{\sqrt{n}} \left\{  \left((\chi^j)^*(n) - \chi^j(n)\right)\, \widehat{h}\!\left( \frac{\log n}{2\pi} \right) + \overline{\big((\chi^j)^*(n) - \chi^j(n)\big)}\,\widehat{h}\!\left( -\frac{\log n}{2\pi} \right)  \right\}. \nonumber
\end{align}
For each $j$, let us call the left-hand side of \eqref{20230323_10:21} by ${\rm (LHS)}_j$ and the right-hand side of \eqref{20230323_10:21} by ${\rm (RHS)}_j$. The idea is to sum \eqref{20230323_10:21} over $j = 1,2,\ldots, \ell -1$  and proceed with an asymptotic analysis as $q \to \infty$.

\subsection{Asymptotic analysis} We analyze the right-hand and the left-hand sides of \eqref{20230323_10:21} separately.

\subsubsection{Analysis of \,$\sum_{j=1}^{\ell-1}{\rm (RHS)}_j$} \label{RHS_THM1} Since $\widehat{F}$ is real-valued, note that $\widehat{h}$ is real-valued and even. Throughout the rest of the proof below, $p$ denotes a prime number. First note that, for each $j$, by applying Lemma \ref{lemNonPrimChar} to $\chi^j$ and $\widehat{h}$, we have
\begin{align}
& \left|\frac{1}{2\pi}  \sum_{n\geq2} \frac{\Lambda(n)}{\sqrt{n}} \left\{  \left((\chi^j)^*(n) - \chi^j(n)\right)\, \widehat{h}\!\left( \frac{\log n}{2\pi} \right) + \overline{\big((\chi^j)^*(n) - \chi^j(n)\big)}\,\widehat{h}\!\left( -\frac{\log n}{2\pi} \right)  \right\}\right|   = O_F\left( \sqrt{\log q}\right). \label{20230412_21:33}
\end{align}
Recall the orthogonality relation
\begin{align*}
\sum_{j = 1}^{\ell -1} \chi^j(n) = 
\begin{cases}
\ell -1\,, & {\rm if} \ \ \chi(n) = 1;\\
-1\,, &{\rm if} \ \ \chi(n) \neq 0, 1.
\end{cases}
\end{align*}
Letting $\chi_0$ be the principal character modulo $q$, and recalling the definition of $n_{\chi}$ and the fact that $\widehat{h}$ is real-valued and even, we have
\begin{align}
& \sum_{j = 1}^{\ell -1} \frac{1}{2\pi}  \sum_{n\geq2} \frac{\Lambda(n)}{\sqrt{n}} \left\{  \chi^j(n) \, \widehat{h}\!\left( \frac{\log n}{2\pi} \right) + \overline{\chi^j(n)}\,\widehat{h}\!\left( -\frac{\log n}{2\pi} \right)  \right\} \nonumber \\
& = \frac{(\ell \!-\!1)}{\pi} \sum_{\substack{n\geq 2 \\ \chi(n)=1}}\frac{\Lambda(n)}{\sqrt{n}}  \, \widehat{h}\!\left( \frac{\log n}{2\pi} \right) - \frac{1}{\pi} \sum_{\substack{n\geq 2 \\ \chi(n)\neq 0,1}}\frac{\Lambda(n)}{\sqrt{n}}  \, \widehat{h}\!\left( \frac{\log n}{2\pi} \right) \label{20230412_21:37}\\
& = \frac{(\ell \!-\!1)}{\pi} \!\!\sum_{2 \leq n < n_{\chi}}\!\!\!\frac{\Lambda(n)\chi_0(n)}{\sqrt{n}}  \, \widehat{h}\!\left( \frac{\log n}{2\pi} \right)  +  \frac{\ell}{\pi} \sum_{\substack{n\geq n_{\chi} \\ \chi(n)=1}}\frac{\Lambda(n)}{\sqrt{n}}  \, \widehat{h}\!\left( \frac{\log n}{2\pi} \right)- \frac{1}{\pi} \sum_{n\geq n_{\chi}}\frac{\Lambda(n)\chi_0(n)}{\sqrt{n}}  \, \widehat{h}\!\left( \frac{\log n}{2\pi} \right) \nonumber \\
& = \frac{(\ell \!-\!1)}{\pi} \!\! \sum_{2 \leq n < n_{\chi}}\!\!\! \frac{\Lambda(n)}{\sqrt{n}}  \, \widehat{h}\!\left( \frac{\log n}{2\pi} \right)  +  \frac{\ell}{\pi} \sum_{\substack{n\geq n_{\chi} \\ \chi(n)=1}}\frac{\Lambda(n)}{\sqrt{n}}  \, \widehat{h}\!\left( \frac{\log n}{2\pi} \right)- \frac{1}{\pi} \sum_{n\geq n_{\chi}}\frac{\Lambda(n)}{\sqrt{n}}  \, \widehat{h}\!\left( \frac{\log n}{2\pi} \right) + O_F\!\left( \ell \sqrt{\log q}\right) \nonumber \\
& \geq \frac{(\ell \!-\!1)}{\pi}\!\! \sum_{2 \leq n < n_{\chi}}\!\!\! \frac{\Lambda(n)}{\sqrt{n}}  \, \widehat{h}\!\left( \frac{\log n}{2\pi} \right)  -  \frac{\ell}{\pi} \sum_{n\geq n_{\chi} }\frac{\Lambda(n)}{\sqrt{n}}  \, \widehat{h}_-\!\left( \frac{\log n}{2\pi} \right)- \frac{1}{\pi} \sum_{n\geq n_{\chi}}\frac{\Lambda(n)}{\sqrt{n}}  \, \widehat{h}\!\left( \frac{\log n}{2\pi} \right) + O_F\!\left( \ell \sqrt{\log q}\right). \nonumber 
\end{align}
The removal of the principal character $\chi_0$ in the fourth line above at the expense of a small error term is justified as in \eqref{20230325_16:04}. We now analyze the three sums over primes in the last line of \eqref{20230412_21:37}. Applying Lemma \ref{lemSumOverPrimes}(i) to $g = \widehat{h}$ and $m = n_\chi$, we use the known bound $n_\chi \ll \log^2 q$ to obtain 
\begin{align}\label{20230325_17:50}
\begin{split}
\frac{1}{\pi}  \sum_{2 \leq n < n_{\chi}}\frac{\Lambda(n)}{\sqrt{n}}   \, \widehat{h}\!\left( \frac{\log n}{2\pi} \right)  
 &= 2  \int_0^{\Delta} \widehat{h}(t) \, e^{\pi t} \,\d t + O_F\!\left((\log \log q)^2\right) \,\\
& = \int_0^{\Delta} \left(\widehat{F}(t - \Delta) + \widehat{F}(-t - \Delta)\right)  \, e^{\pi t} \,\d t  + O_F\!\left((\log \log q)^2\right) \\
&= \int_0^{\Delta}\widehat{F}(t - \Delta)\, e^{\pi t} \,\d t  + O_F\!\left((\log \log q)^2\right)
\\& = e^{\pi \Delta} \int_{-\infty}^{0}\widehat{F}(y) \, e^{\pi y} \,\d y\, + O_F\!\left((\log \log q)^2\right),
\end{split}
\end{align}
from the assumptions that $\supp\big( \widehat{F}\big) \subset [-N, N]$ and $\Delta \geq N$. Now applying Lemma \ref{lemSumOverPrimes}(ii), we similarly get 
\begin{align}\label{20230325_17:52}
\frac{1}{\pi}  \sum_{n\geq n_{\chi} }\frac{\Lambda(n)}{\sqrt{n}}  \, \widehat{h}_-\!\left( \frac{\log n}{2\pi} \right) = e^{\pi \Delta} \int_{0}^{\infty}\widehat{F}_{-}(y) \, e^{\pi y} \,\d y + O_F\!\left((\log \log q)^2\right) \,,
\end{align}
and
\begin{align}\label{20230412_21:46}
\frac{1}{\pi}  \sum_{n\geq n_{\chi} }\frac{\Lambda(n)}{\sqrt{n}}  \, \widehat{h}\!\left( \frac{\log n}{2\pi} \right) = e^{\pi \Delta} \int_{0}^{\infty}\widehat{F}(y) \, e^{\pi y} \,\d y + O_F\!\left((\log \log q)^2\right).
\end{align}

\smallskip

Plugging \eqref{20230325_17:50}, \eqref{20230325_17:52}, and \eqref{20230412_21:46} back into \eqref{20230412_21:37}, and recalling \eqref{20230412_21:33}, we arrive at 
\begin{align}\label{20230412_22:32}
\begin{split}
& \sum_{j=1}^{\ell-1}{\rm (RHS)}_j  \geq e^{\pi \Delta}\left( (\ell \! - \! 1) \int_{-\infty}^{0}\widehat{F}(y) \, e^{\pi y} \,\d y -  \int_{0}^{\infty}\big(\ell \,\widehat{F}_{-}(y) + \widehat{F}(y)\big) \, e^{\pi y} \,\d y\right) + O_F\!\left( \ell\sqrt{\log q}\right)\\
&= e^{\pi \Delta}(\ell \!- \!1) \left(   \int_{-\infty}^{0}\widehat{F}(y) \, e^{\pi y} \,\d y -   \int_{0}^{\infty}\widehat{F}_-(y)\,  e^{\pi y} \,\d y - \frac{1}{(\ell\!-\!1)} \int_{0}^{\infty}\widehat{F}_+(y) \, e^{\pi y} \,\d y \right)+ O_F\!\left( \ell\sqrt{\log q}\right).
\end{split}
\end{align}

\subsubsection{Analysis of \,$\sum_{j=1}^{\ell-1}{\rm (LHS)}_j$} \label{LHS_THM_SUB} Start by noting that $\widehat{h}(0) = 0$, since $\supp\big( \widehat{F}\big) \subset [-N, N]$ and $\Delta \geq N$. Note also that 
\begin{equation}\label{20230327_10:58}
h(u) = \tfrac12\big(e^{2\pi i u \Delta} F(u) + e^{-2 \pi i u \Delta}F(-u)\big). 
\end{equation} 
Using Stirling's formula for $\Gamma'/\Gamma$, we get 
\begin{align}\label{20230327_13:19}
\left|\int_{-\infty}^{\infty} h(u) \, \mathrm{Re} \,  \frac{\Gamma'}{\Gamma}\!\left(\frac{2 - (\chi^j)^*(-1)}{4} + \frac{iu}{2}\right)  \mathrm{d}u\right| \ll\int_{-\infty}^{\infty} |h(u)| \,\log(2 + |u|)\,\d u = O_F(1).
\end{align}
It remains to analyze the sum over the zeros in the left hand-sides of \eqref{20230323_10:21}, and we are interested in an upper bound for it. This is provided by Lemma \ref{lemSumOverZeros}, which we apply to $|F|$ and $(\chi^j)^*$ to get 
\begin{align}
\label{20250710_11:39}
\begin{split}
\left|\sum_{\gamma_{(\chi^j)^*}} h\!\left(\gamma_{(\chi^j)^*}\right) \right| &\leq \sum_{\gamma_{(\chi^j)^*}} \left|F\!\left(\gamma_{(\chi^j)^*}\right)\right|\\
&\leq \frac{\log q_j}{2\pi} \,  \|F\|_1+ O_F\!\left( \frac{\log q_j}{\log \log q_j}\right) \\
&\leq \frac{\log q}{2\pi} \,  \|F\|_1 + O_F\!\left( \frac{\log q}{\log \log q}\right).
\end{split}
\end{align}
From \eqref{20230327_13:19} and \eqref{20250710_11:39}, we arrive at 
\begin{equation}\label{20230327_13:22}
\sum_{j=1}^{\ell-1}{\rm (LHS)}_j \leq \frac{(\ell\! -\!1)\log q}{2\pi}  \, \|F\|_1 + O_F\!\left(  \frac{\ell\, \log q}{\log \log q}\right).
\end{equation}

\subsection{Conclusion} \label{conc} Summing \eqref{20230323_10:21} over $j = 1,2,\ldots, \ell -1$, and using \eqref{20230412_22:32} and \eqref{20230327_13:22} we get 
\begin{align*}
e^{\pi \Delta} \left(   \int_{-\infty}^{0}\widehat{F}(y) \, e^{\pi y} \,\d y -   \int_{0}^{\infty}\widehat{F}_-(y)  \,e^{\pi y} \,\d y - \frac{1}{(\ell\!-\!1)} \int_{0}^{\infty}\widehat{F}_+(y)  \,e^{\pi y} \,\d y \right) \leq \frac{\log q}{2\pi} \, \|F\|_1 + O_F\!\left( \frac{\log q}{\log \log q}\right).
\end{align*}
Recalling that $n_{\chi} = e^{2\pi \Delta}$, this yields
\begin{align}\label{20230327_13:31}
\limsup_{q \to \infty}\frac{\sqrt{n_{\chi}}}{\log q} \leq \frac{1}{2\pi} \frac{\|F\|_1}{ \left(   \int_{-\infty}^{0}\widehat{F}(y) \, e^{\pi y} \,\d y -   \int_{0}^{\infty}\widehat{F}_-(y) \, e^{\pi y} \,\d y - \frac{1}{(\ell-1)} \int_{0}^{\infty}\widehat{F}_+(y) \, e^{\pi y} \,\d y \right)}\,,
\end{align}
where we assume that the denominator on the right-hand side of \eqref{20230327_13:31} is positive. At this stage we can take the infimum of the right-hand side of \eqref{20230327_13:31} over $F \in \mc{A}$ with $\widehat{F} \in C^{\infty}_c(\R)$ and, by Theorem \ref{Thm1} (i), such an infimum is indeed $\mc{C} \!\left(  \frac{1}{(\ell-1)}\right)^{-1}$. This concludes the proof of Theorem \ref{Thm3}.

\subsection{Proof of Theorem \ref{Thm_LPQR}} We now indicate how to appropriately modify the ideas in the proof of Theorem \ref{Thm3} in order to prove Theorem \ref{Thm_LPQR}. Suppose that $p$ is prime and that $\chi$ is the Legendre symbol modulo $p$. Then $\chi$ is primitive and, assuming GRH, Lemma \ref{GW_lemma} yields
\[
\begin{split}
- \frac{1}{2\pi}  \sum_{n\geq2}&\frac{\Lambda(n) \chi(n)}{\sqrt{n}} \left\{  \widehat{h}\!\left( \frac{\log n}{2\pi} \right) + \widehat{h}\!\left( -\frac{\log n}{2\pi} \right)  \right\} 
\\
&\qquad= \sum_{\rho_{\chi}} h( \gamma_\chi) - \widehat{h}(0) \frac{\log (p/\pi)}{2\pi}   -  \frac{1}{2\pi}  \int_{-\infty}^{\infty} h(u) \, \mathrm{Re} \,  \frac{\Gamma'}{\Gamma}\!\left(\frac{2 - \chi(-1)}{4} + \frac{iu}{2}\right)  \mathrm{d}u.
\end{split}
\]
Now set $\Delta:=\log r_{p}/(2\pi)$ so that $r_{p} = e^{2\pi \Delta}$, and let $\widehat{h}(t)$ be defined as in \eqref{20230323_09:59}. As before, since $\widehat h \in C_c^\infty (\R)$, the Paley-Wiener theorem \cite[Theorem IX.11]{RS} tells us that $h$ extends to an analytic function of $\C$, which in the strip $|\mathrm{Im} \, s| \le 1$ has the decay $|h(s)| + |h'(s)| \ll (1+|s|)^{-(1+\delta)}$ for some $\delta>0$ when $|\mathrm{Re} \, s| \to \infty$. So that we can indeed apply Lemma \ref{GW_lemma} to this choice of $h$.

With these choices, we still have $\widehat{h}(0)=0$, the bound in \eqref{20230327_13:19} still holds, and Lemma \ref{lemSumOverZeros} gives
\[
\left|\sum_{\gamma_{\chi}} h\!\left(\gamma_{\chi}\right) \right| \leq \sum_{\gamma_{\chi}} \left|F\!\left(\gamma_{\chi}\right)\right|\leq \frac{\log p}{2\pi}  \|F\|_1 + O_F\!\left( \frac{\log p}{\log \log p}\right).
\]
For each prime $\ell < r_p$, we have $\chi(\ell) =-1$, hence applying Lemma \ref{lemSumOverPrimes} and using Ankeny's result that $r_p =O(\log^2p)$ assuming GRH, we deduce that 
\[
\begin{split}
- \frac{1}{2\pi}  \sum_{n\geq2}  \frac{\Lambda(n) \chi(n)}{\sqrt{n}} & \left\{  \widehat{h}\!\left( \frac{\log n}{2\pi} \right) + \widehat{h}\!\left( -\frac{\log n}{2\pi} \right)  \right\}  
\\
&= - \frac{1}{\pi}  \sum_{n\geq2}\frac{\Lambda(n) \chi(n)}{\sqrt{n}}  \widehat{h}\!\left( \frac{\log n}{2\pi} \right) 
\\
&\ge \frac{1}{\pi} \sum_{n < r_p} \frac{\Lambda(n)}{\sqrt{n}}  \widehat{h}\!\left( \frac{\log n}{2\pi} \right) - \frac{1}{\pi} \sum_{n \ge r_p} \frac{\Lambda(n)}{\sqrt{n}}  \left|\widehat{h}\!\left( \frac{\log n}{2\pi} \right)\right| 
+ O_F( \log \log p)
\\
&=2  \int_0^{\Delta} \widehat{h}(t) \, e^{\pi t} \,\d t - 2  \int_\Delta^{\infty} \left|\widehat{h}(t)\right| \, e^{\pi t} \,\d t+O_F\!\left((\log \log p)^2\right) 
\\
&= e^{\pi \Delta} \int_{-\infty}^{0} \widehat{F}(t)\,e^{\pi t} \,\d t - e^{\pi \Delta}  \int_{0}^{\infty} \big|\widehat{F}(t)\big|\,e^{\pi t} \,\d t 
+ O_F\big( ( \log \log p)^2\big).
\end{split}
\]
Combining estimates, we derive that
\begin{align*}
e^{\pi \Delta} \left(   \int_{-\infty}^{0}\widehat{F}(y) \, e^{\pi y} \,\d y -   \int_{0}^{\infty}|\widehat{F}(y)|  \,e^{\pi y} \,\d y\right) \leq \frac{\log p}{2\pi} \, \|F\|_1 + O_F\!\left( \frac{\log p}{\log \log p}\right).
\end{align*}
Rearranging and recalling that $r_{p} = e^{2\pi \Delta}$, we send $p \to \infty$ to obtain
\begin{align*}
\limsup_{p \to \infty}\frac{\sqrt{r_{p}}}{\log p} \leq \frac{1}{2\pi} \frac{\|F\|_1}{ \left(   \int_{-\infty}^{0}\widehat{F}(y) \, e^{\pi y} \,\d y -   \int_{0}^{\infty}|\widehat{F}(y)| \, e^{\pi y} \,\d y \right)}\,.
\end{align*}
We now take the infimum of the right-hand side above over $F \in \mc{A}$ with $\widehat{F} \in C^{\infty}_c(\R)$. This, in view of Theorem \ref{Thm1} (i), yields
\[
\limsup_{p \to \infty}\frac{\sqrt{r_{p}}}{\log p} \leq \frac{1}{\mathcal{C}(1)}.
\]
Theorem \ref{Thm_LPQR} now follows upon squaring both sides of this inequality.

\section{Primes in arithmetic progressions: proof of Theorem \ref{Thm4}} \label{Section_primes_AP}

\subsection{Setup} Throughout the proof below we set 
$$\Delta:=\log P(a,q)/(2\pi),$$ 
and hence $P(a,q) = e^{2\pi \Delta}$. We may assume without loss of generality that $q$ is large and that 
\begin{equation*}
\frac12 (\phi(q) \log q)^2\leq P(a,q) \leq(\phi(q) \log q)^2.
\end{equation*}
The upper bound is true under GRH, for all $q >3$, from the work of Lamzouri, Li and Soundararajan \cite{LLS} and, for the cases where the lower bound does not hold, our Theorem \ref{Thm4} is trivially true. In particular we have, as $q \to \infty$,
\begin{equation}\label{20230328_18:47}
\pi \Delta =  \log \phi(q) + O(\log \log q) = \big(1 + o(1)\big) \log q.
\end{equation}

As before, let $F \in \mc{A}$ with $\widehat{F} \in C^{\infty}_c(\R)$ be fixed, say with $\supp\big( \widehat{F}\big) \subset [-N, N]$, and define $\widehat{h}$ by \eqref{20230323_09:59}. Assume without loss of generality that $\Delta \geq N$. From the orthogonality relations between the Dirichlet characters modulo $q$ we have 
\begin{align}\label{20230327_14:44}
\begin{split}
\phi(q)\!\!\! \!\!\!\sum_{\substack{n\geq2 \\ n \equiv a\,  ({\rm mod} \, q)}} & \frac{\Lambda(n)}{\sqrt{n}}   \, \widehat{h}\!\left( \frac{\log n}{2\pi} \right) = \sum_{n\geq2} \frac{\Lambda(n)\chi_0(n)}{\sqrt{n}}   \, \widehat{h}\!\left( \frac{\log n}{2\pi} \right) + \sum_{\chi \neq \chi_0} \overline{\chi(a)}\sum_{n\geq2} \frac{\Lambda(n)\chi(n)}{\sqrt{n}}   \, \widehat{h}\!\left( \frac{\log n}{2\pi} \right)\,,
\end{split}
\end{align}
where $\chi_0$ is the principal character modulo $q$. The strategy  is to proceed with an asymptotic analysis of each of the three sums in \eqref{20230327_14:44}. 
Part of the required work for the two sums on the right-hand side of \eqref{20230327_14:44} parallels that of Section \ref{Sec3_CNR}.
For the sum on the left-hand side of \eqref{20230327_14:44}, the idea is to use the definition of $P(a,q)$ together with the Brun-Titchmarsh inequality to provide a suitable upper bound.

\subsection{Asymptotic analysis} 
\subsubsection{The sum with the principal character} We can reduce matters to the Riemann zeta-function:
\begin{align*}
\sum_{n\geq2} \frac{\Lambda(n)\chi_0(n)}{\sqrt{n}}   \, \widehat{h}\!\left( \frac{\log n}{2\pi} \right) & = \sum_{n\geq2} \frac{\Lambda(n)}{\sqrt{n}}   \, \widehat{h}\!\left( \frac{\log n}{2\pi} \right) - \sum_{p|q}\sum_{k\geq 1} \frac{\log p}{p^{k/2}}  \ \widehat{h}\!\left( \frac{\log p^k}{2\pi} \right) \nonumber \\
& = \sum_{n\geq2} \frac{\Lambda(n)}{\sqrt{n}}   \, \widehat{h}\!\left( \frac{\log n}{2\pi} \right) + O_F\left(\sqrt{\log q}\right),
\end{align*}
where the error term comes from the estimate
\begin{align*}
    \sum_{p|q}\sum_{k\geq 1} \frac{\log p}{p^{k/2}} &= \sum_{p |q} \frac{\log p}{\sqrt{p}} + \sum_{p|q} \frac{\log p}{p}  + \sum_{p|q}\sum_{k\geq 3 }\frac{\log p}{p^{k/2}} \\
        &\leq 2 \sum_{p |q} \frac{\log p}{\sqrt{p}}+  \sum_{p}\sum_{k\geq 3 }\frac{\log p}{p^{k/2}} \\
        &\ll \sqrt{\log q},
\end{align*}
obtained as in \eqref{20230325_16:04}.

By the bound $P(a,q) \leq (\phi(q) \log q)^2$ and items (i) and (ii) of Lemma \ref{lemSumOverPrimes}, we have
\begin{align*}
     \sum_{n\geq2} \frac{\Lambda(n)}{\sqrt{n}}   \, \widehat{h}\!\left( \frac{\log n}{2\pi} \right) =  2 \pi  \int_{0}^{\infty}\widehat{h}(y) \, e^{\pi y} \,\d y +  O_F\big(\log^2q\big) = \pi e^{\pi \Delta} \int_{-\infty}^{\infty}\widehat{F}(y) \, e^{\pi y} \,\d y + O_F\big(\log^2q\big).
\end{align*} 
In conclusion, we have shown that
\begin{equation*}\label{20230327_17:22}
\sum_{n\geq2} \frac{\Lambda(n)\chi_0(n)}{\sqrt{n}}   \, \widehat{h}\!\left( \frac{\log n}{2\pi} \right) = \pi e^{\pi \Delta} \int_{-\infty}^{\infty}\widehat{F}(y) \, e^{\pi y} \,\d y + O_F\big(\log^2q\big).
\end{equation*}

\subsubsection{The sum over the non-principal characters} For each $\chi \neq \chi_0$ let us define the function $h_{\chi}$ by 
$$\widehat{h_{\chi}}(t):=\frac{\overline{\chi(a)}\,\widehat{F}(t - \Delta) + \chi(a)\,\widehat{F}(-t - \Delta)}{2}.$$
Since $\supp\big( \widehat{F}\big) \subset [-N, N]$ and $\Delta \geq N$, note that $\widehat{h_{\chi}}(0) = 0$ and 
\begin{align*}
\widehat{h_{\chi}}(t) = 
\begin{cases}
\overline{\chi(a)} \, \widehat{h}(t), & {\rm for} \ t \geq 0;\\
\chi(a) \, \widehat{h}(t),& {\rm for} \ t \leq 0.
\end{cases}
\end{align*}
Since the sum is real-valued and $\widehat{h}$ is even, we have
\begin{align}\label{20230327_17:15}
\begin{split}
\sum_{\chi \neq \chi_0} & \overline{\chi(a)}\sum_{n\geq2} \frac{\Lambda(n)\chi(n)}{\sqrt{n}}   \, \widehat{h}\!\left( \frac{\log n}{2\pi} \right)   \\
&  = \frac12 \sum_{\chi \neq \chi_0} \sum_{n\geq2} \frac{\Lambda(n)}{\sqrt{n}} \left\{\chi(n)  \overline{\chi(a)} \, \widehat{h}\!\left( \frac{\log n}{2\pi} \right)+ \overline{\chi(n)}\chi(a)\, \widehat{h}\!\left( -\frac{\log n}{2\pi} \right)\right\}\\
& = \frac12 \sum_{\chi \neq \chi_0} \sum_{n\geq2} \frac{\Lambda(n)}{\sqrt{n}}\left\{\chi(n) \, \widehat{h_{\chi}}\!\left( \frac{\log n}{2\pi} \right) + \overline{\chi(n)}\, \widehat{h_{\chi}}\!\left(- \frac{\log n}{2\pi} \right)\right\}. 
\end{split}
\end{align}

If $\chi$ is a non-principal Dirichlet character modulo $q$, and $\chi^*$ modulo $q^*$ is the unique primitive Dirichlet character that induces $\chi$ (we include the possibility of $\chi^* = \chi$, if $\chi$ is primitive) we rewrite
\begin{align}\label{20230328_10:48}
\begin{split}
\sum_{n\geq2} & \frac{\Lambda(n)}{\sqrt{n}}\left\{\chi(n) \, \widehat{h_{\chi}}\!\left( \frac{\log n}{2\pi} \right) + \overline{\chi(n)}\, \widehat{h_{\chi}}\!\left(- \frac{\log n}{2\pi} \right)\right\} \\
& = \sum_{n\geq2} \frac{\Lambda(n)}{\sqrt{n}}\left\{\chi^*(n) \, \widehat{h_{\chi}}\!\left( \frac{\log n}{2\pi} \right) + \overline{\chi^*(n)}\, \widehat{h_{\chi}}\!\left(- \frac{\log n}{2\pi} \right)\right\} \\
& \ \ \ \ \ \ \ \ \ + \sum_{n\geq2} \frac{\Lambda(n)}{\sqrt{n}}\left\{\big(\chi(n) - \chi^*(n)\big) \, \widehat{h_{\chi}}\!\left( \frac{\log n}{2\pi} \right) + \big(\overline{\chi(n)} - \overline{\chi^*(n)}\big)\, \widehat{h_{\chi}}\!\left(- \frac{\log n}{2\pi} \right)\right\}\\
& = \sum_{n\geq2} \frac{\Lambda(n)}{\sqrt{n}}\left\{\chi^*(n) \, \widehat{h_{\chi}}\!\left( \frac{\log n}{2\pi} \right) + \overline{\chi^*(n)}\, \widehat{h_{\chi}}\!\left(- \frac{\log n}{2\pi} \right)\right\}  + O_F\!\left(\sqrt{\log q}\right)\,,
\end{split}
\end{align}
with the error term coming from an application of Lemma \ref{lemNonPrimChar}. Note that 
\begin{equation*}
h_{\chi}(u) = \frac12\Big(\overline{\chi(a)} e^{2\pi i u \Delta} F(u) + \chi(a)e^{-2 \pi i u \Delta}F(-u)\Big)
\end{equation*} 
is real-valued.  Note also that since $\widehat {h_\chi} \in C_c^\infty (\R)$, the Paley-Wiener theorem \cite[Theorem IX.11]{RS} tells us that $h_\chi$ extends to an analytic function of $\C$, which in the strip $|\mathrm{Im} \, s| \le 1$ has the decay  $|h_\chi(s)|+ |h_\chi'(s)| \ll (1+|s|)^{-(1+\delta)}$ for some $\delta>0$ when $|\mathrm{Re} \, s| \to \infty$. So now we use the explicit formula (Lemma \ref{GW_lemma}) for the primitive character $\chi^*$, and proceed by applying Stirling's formula for $\Gamma'/\Gamma$ and Lemma \ref{lemSumOverZeros} to get 
\begin{align}\label{20230327_17:16}
\begin{split}
& \left|\sum_{n\geq2} \frac{\Lambda(n)}{\sqrt{n}}\left\{\chi^*(n) \, \widehat{h_{\chi}}\!\left( \frac{\log n}{2\pi} \right) + \overline{\chi^*(n)}\, \widehat{h_{\chi}}\!\left(- \frac{\log n}{2\pi} \right)\right\}\right| \\
&  \ \ \ \ \ \ = \left| - 2\pi  \sum_{\gamma_{\chi^*}} h_{\chi}\!\left(\gamma_{\chi^*}\right) + \int_{-\infty}^{\infty} h_{\chi}(u) \, \mathrm{Re} \,  \frac{\Gamma'}{\Gamma}\left(\frac{2 - \chi^*(-1)}{4} + \frac{iu}{2}\right)  \mathrm{d}u \right|  \\
&\ \ \ \ \ \ \leq  2\pi  \sum_{\gamma_{\chi^*}} |F\!\left(\gamma_{\chi^*}\right)| + O\left( \int_{-\infty}^{\infty} |h_\chi(u)| \,\log(2 + |u|)\,\d u \right) \\
&\ \ \ \ \ \ \leq (\log q) \, \|F\|_1 +O_F\!\left( \frac{\log q}{\log \log q}\right).  
\end{split}
\end{align}
From \eqref{20230327_17:15}, \eqref{20230328_10:48}, \eqref{20230327_17:16}, and the triangle inequality, we arrive at 
\begin{align}\label{20230327_17:24}
\left|\sum_{\chi \neq \chi_0}  \overline{\chi(a)}\sum_{n\geq2} \frac{\Lambda(n)\chi(n)}{\sqrt{n}}   \, \widehat{h}\!\left( \frac{\log n}{2\pi} \right)\right|  \leq \frac{\phi(q)\log q}{2}\, \|F\|_1 + O_F\!\left( \frac{\phi(q)\log q}{\log \log q}\right).
\end{align}

\subsubsection{The sum over the arithmetic progression: an upper bound via the Brun-Titchmarsh inequality} We continue to reserve below the letter $p$ for a prime number. From the definition of $\widehat{h}$ and $P(a,q)$, we have
\begin{align}\label{20230328_17:03}
\begin{split}
\sum_{\substack{n\geq2 \\ n \equiv a\,  ({\rm mod} \, q)}}  \frac{\Lambda(n)}{\sqrt{n}}   \, \widehat{h}\!\left( \frac{\log n}{2\pi} \right) & = \frac12 \sum_{\substack{2 \leq n < e^{2\pi \cdot (\Delta +N)} \\ n \equiv a\,  ({\rm mod} \, q)}}  \frac{\Lambda(n)}{\sqrt{n}}   \, \widehat{F}\!\left( \frac{\log n}{2\pi} - \Delta\right)\\
& = \frac12 \sum_{\substack{e^{2\pi \Delta} \leq p < e^{2\pi \cdot(\Delta +N)} \\ p \equiv a\,  ({\rm mod} \, q)}}  \frac{\log p}{\sqrt{p}}   \, \widehat{F}\!\left( \frac{\log p}{2\pi} - \Delta\right) + O_F(\log q).
\end{split}
\end{align}
The error term above follows from the fact that 
\begin{equation*}
\sum_p\sum_{ k\geq 3 }\frac{\log p}{p^{k/2}} \ll 1 \quad \text{ and } \quad  \sum_{p \leq \sqrt{e^{2\pi \cdot(\Delta +N)}}} \frac{\log p}{p} \ll  \pi \Delta  \ll  \log q.
\end{equation*}
We now break the sum on the right-hand side of \eqref{20230328_17:03} into intervals $I_k$ of size $|I_k| = \frac{e^{2\pi \Delta} }{\Delta}$ by writing
\begin{align*}
\frac12 \sum_{\substack{e^{2\pi \Delta} \leq p < e^{2\pi \cdot(\Delta +N)} \\ p \equiv a\,  ({\rm mod} \, q)}}  \frac{\log p}{\sqrt{p}}   \, \widehat{F}\!\left( \frac{\log n}{2\pi} - \Delta\right) = \frac12 \sum_{k=0}^{M}\sum_{\substack{p \in I_k  \\ p \equiv a\,  ({\rm mod} \, q)}}  \frac{\log p}{\sqrt{p}}   \, \widehat{F}\!\left( \frac{\log p}{2\pi} - \Delta\right)\,,
\end{align*}
where $I_k :=\Big[ e^{2\pi \Delta} + k \frac{e^{2\pi \Delta} }{\Delta}\,,\, e^{2\pi \Delta} + (k+1) \frac{e^{2\pi \Delta} }{\Delta}\Big)$ and $M :=\lfloor\Delta(e^{2 \pi N} -1)\rfloor$.

\smallskip

Letting $\pi(x,q,a)$ be the number of primes $p \leq x$ with $p \equiv a\, ({\rm mod}\, q)$, Montgomery and Vaughan \cite[Theorem 2]{MV_LS} established the following version of the Brun-Titchmarsh inequality: for any real numbers $x>0$ and $y >q$, one has
\begin{align}\label{20230328_17:26}
\pi(x+y, q,a) - \pi(x,q,a) < \frac{2y}{\phi(q) \log(y/q)}.
\end{align}
Let us momentarily shorten the notation and write 
\begin{equation}\label{20230328_21:29}
G(x) := \frac{1}{\sqrt{x}}\,\widehat{F}\left( \frac{\log x}{2\pi} - \Delta\right).
\end{equation}
Using that $\log p \leq  2\pi\cdot(\Delta +N)$ in our range, inequality \eqref{20230328_17:26} plainly leads us to 
\begin{align*}
\sum_{\substack{p \in I_k  \\ p \equiv a\,  ({\rm mod} \, q)}}  (\log p)  \, G(p)  \leq \left(\sup_{x \in I_k} G_+(x)\right) \big( 2\pi\cdot(\Delta +N)\big)\frac{2 |I_k|}{\phi(q) \log (|I_k|/q)}\,,
\end{align*}
and hence 
\begin{align}\label{20230328_19:35}
\begin{split}
\frac12 \sum_{k=0}^{M}\sum_{\substack{p \in I_k  \\ p \equiv a\,  ({\rm mod} \, q)}}  (\log p)   \, G(p) & \leq \left(\frac{2\pi\cdot(\Delta +N)}{\phi(q)\, \log (|I_k|/q)} \right)\sum_{k=0}^{M}\left(\sup_{x \in I_k} G_+(x)\right)|I_k| \\
& = \frac{\big(2 + o_F(1)\big)}{\phi(q)} \sum_{k=0}^{M}\left(\sup_{x \in I_k} G_+(x)\right)|I_k|
\end{split}
\end{align}
as $q \to \infty$, where we have used \eqref{20230328_18:47}. 

\smallskip

Let $I = \cup_{k=0}^M I_k$. The last sum in \eqref{20230328_19:35} is a Riemann sum, and the idea is to compare with the integral of the function $G_+(x)$ over $I$. Let $x _k \in \overline{I_k}$ be such that $\sup_{x \in I_k} G_+(x) = G_+(x_k)$. The functions $\widehat{F}_+$ and $G_+$ are absolutely continuous and hence we have the classical derivative
\begin{align*}
(G_+)'(t) = -\frac{1}{2 \, t^{3/2}} \,\widehat{F}_+\!\!\left( \frac{\log t}{2\pi} - \Delta\right)  + \frac{1}{2\pi \, t^{3/2}} \, \big(\widehat{F}_+)'\!\left( \frac{\log t}{2\pi} - \Delta\right).
\end{align*}
for almost every $t$. Since $t \geq e^{2 \pi \Delta}$ in our interval $I$, we get
\begin{align}\label{20240405_10:24}
\big|(G_+)'(t) \big| \ll_F \frac{1}{e^{3\pi \Delta}}  \ \ \ {\rm for \ a.e.}   \ t \in I.
\end{align}
Then, if $x \in I_k$, we use the fundamental theorem of calculus and \eqref{20240405_10:24} to get
\begin{align}\label{20230328_20:48}
\big|G_+(x_k) - G_+(x)\big|= \left| \int_x^{x_k} (G_+)'(t)\,\d t\, \right| \ll_F \frac{|I_k|}{e^{3\pi \Delta}}.
\end{align}
Using \eqref{20230328_20:48}, we get
\begin{align*}
&  \left|\sum_{k=0}^{M} \left(\sup_{x \in I_k} G_+(x)\right)|I_k| - \int_I G_+(x)\,\d x\right|    = \left| \sum_{k=0}^{M} \int_{I_k} \big(G_+(x_k) - G_+(x) \big)\,\d x \right| \\
 &  \ \ \ \ \ \ \ \leq  \sum_{k=0}^{M} \int_{I_k} \big|G_+(x_k) - G_+(x) \big|\,\d x  \\
 & \ \ \ \ \ \ \  \ll_F (M+1) \, \frac{|I_k|^2}{e^{3\pi \Delta}} \leq \frac{e^{2\pi N} e^{\pi \Delta}}{\Delta}\,,
\end{align*}
and therefore, as $q \to \infty$, we have
\begin{align}\label{20230328_21:24}
\begin{split}
\sum_{k=0}^{M} \left(\sup_{x \in I_k} G_+(x)\right)|I_k| & \leq \int_I G_+(x)\,\d x + O_F\!\left(\frac{e^{\pi \Delta}}{\Delta}\right) \\
& = 2\pi e^{\pi \Delta} \left(\int_0^{\infty} \widehat{F}_+(y)\, e^{\pi y} \,\d y  + o_F(1)\right),
\end{split}
\end{align}
after an appropriate change of variables and the use of \eqref{20230328_18:47}. 

\smallskip

Combining \eqref{20230328_17:03}, \eqref{20230328_21:29}, \eqref{20230328_19:35}, and \eqref{20230328_21:24} we arrive at 
\begin{align}\label{20230328_21:39}
\sum_{\substack{n\geq2 \\ n \equiv a\,  ({\rm mod} \, q)}}  \frac{\Lambda(n)}{\sqrt{n}}   \, \widehat{h}\!\left( \frac{\log n}{2\pi} \right) \leq \frac{4 \pi \, e^{\pi \Delta}}{\phi(q)}  \left(\int_0^{\infty} \widehat{F}_+(y)\, e^{\pi y} \,\d y  + o_F(1)\right),
\end{align}
as $q \to \infty$.

\subsection{Conclusion} From \eqref{20230327_14:44} we have 
\begin{align*}
\sum_{n\geq2} \frac{\Lambda(n)\chi_0(n)}{\sqrt{n}}   \, \widehat{h}\!\left( \frac{\log n}{2\pi} \right)  - \,\, \phi(q)\!\!\! \!\!\sum_{\substack{n\geq2 \\ n \equiv a\,  ({\rm mod} \, q)}}  \frac{\Lambda(n)}{\sqrt{n}}   \, \widehat{h}\!\left( \frac{\log n}{2\pi} \right)  \leq  \left|\sum_{\chi \neq \chi_0} \overline{\chi(a)}\sum_{n\geq2} \frac{\Lambda(n)\chi(n)}{\sqrt{n}}   \, \widehat{h}\!\left( \frac{\log n}{2\pi} \right)\right|\,,
\end{align*}
and then \eqref{20230327_17:22}, \eqref{20230327_17:24}, and \eqref{20230328_21:39} imply that, as $q \to \infty$,
\begin{align}\label{20230328_21:47}
\pi e^{\pi \Delta}\left(  \int_{-\infty}^{\infty}\widehat{F}(y) \, e^{\pi y} \,\d y - 4 \int_0^{\infty} \widehat{F}_+(y)\, e^{\pi y} \,\d y  + o_F(1)\right)\leq  \frac{\phi(q)\log q}{2}\, \big(\|F\|_1 + o_F(1)\big).
\end{align}
Recalling that $P(a,q) = e^{2\pi \Delta}$, when we send $q \to \infty$ in \eqref{20230328_21:47} we arrive at the inequality 
\begin{align}\label{20230328_22:02}
\limsup_{q \to \infty} \frac{\sqrt{P(a,q)}}{ \phi(q) \log q} \leq \frac{1}{2\pi} \frac{\|F\|_1}{\left(\int_{-\infty}^{0}\widehat{F}(y) \, e^{\pi y} \,\d y - \int_{0}^{\infty}\widehat{F}_-(y) \, e^{\pi y} \,\d y - 3\int_{0}^{\infty}\widehat{F}_+(y) \, e^{\pi y} \,\d y\right)}\,,
\end{align}
where we assume that the denominator on the right-hand side of \eqref{20230328_22:02} is positive. At this point we can take the infimum of the right-hand side of \eqref{20230328_22:02} over $F \in \mc{A}$ with $\widehat{F} \in C^{\infty}_c(\R)$ and, by Theorem \ref{Thm1} (i), such an infimum is indeed $\mc{C}(3)^{-1}$. This concludes the proof of Theorem \ref{Thm4}.

\section{Computer-assisted techniques: proof of Theorem \ref{Thm2}} \label{Sec5_Computational}

With Proposition \ref{Prop6}, the problem of finding upper bounds for $\mc{C}(A)$ is reduced to finding good test functions for the extremal problem (EP2). Intuitively speaking, if $F$ is a near-extremizer for (EP1), the function $\psi(t) = e^{\pi t} \big(-{\bf 1}_{\{F<0\}}(t) + A\, {\bf 1}_{\{F>0\}}(t)\big)$ tends to be a near-extremizer for (EP2), although it may not lie in the class $\mc{B}_A$. Motivated by this intuition, our choices of test functions for (EP2) will be suitable truncations of these, namely 
\begin{align}\label{eq:psi-test}
\psi(t) = e^{\pi t}\sum_{n=0}^{N-1} \left(\left(\tfrac{-1 + (-1)^n}{2}\right) + A \left(\tfrac{1 + (-1)^n}{2}\right)\right) {\bf 1}_{(T_n, T_{n+1})}\, ,
\end{align}
with $0 = T_0 < T_1 < T_2 < \ldots < T_N$. 

\subsection{Upper bound numerics} When taking this family of test functions, (EP2) is now a restricted minimization problem on $\R^N$, over the parameters $(T_1,\ldots,T_N)$. For each given value of $A$, we attempt to solve this problem numerically. Our search routine first takes initial random values of $0 = T_0 < T_1 < T_2 < \ldots < T_N$, and for each random initialization, attempts to find a nearby local minimum with standard numerical optimization methods. This is carried out iteratively, starting with $N=1$ and then using the best example found for a given $N$ to take nearby random initializations for $N+1$, until no significant improvements are found. In this way, for each value of $A$, we find the examples and upper bounds given by Table \ref{tb:uppers}.\smallskip 

\begin{table}[ht]
	\begin{center}
		\begin{tabular}{|c||c|c|c|c|c|c|c|c|}
		\hline
		$A$ & Bound & $T_1$ &$T_2$&$T_3$&$T_4$&$T_5$&$T_6$&$T_7$  \\ \hline \hline
 $\frac{1}{4}$ & 1.33509 & 0.3530083 & 0.3780727 &  0.3925238 & 0.3928645& 0.4072127 & 0.4073054 & - \\ \hline 
 $\frac{1}{3}$ & 1.28781 & 0.3184544 & 0.3597874 & 0.4171521 & 0.4208919 & - & - &- \\ \hline 
  $\frac{1}{2}$ & 1.23080 &0.2490362 & 0.2972170 & 0.3313443 & 0.3330512 & 0.3375152 & - & - \\ \hline 
  $1$ & 1.14731 & 0.1509068& 0.2090402& 0.2318820& 0.2409230& 0.2629189& 0.2789340& \
0.2820042 \\ \hline 
  $3$ & 1.06240 & 0.0561589& 0.1037093& 0.1133532& 0.1234334& 0.1257599& 0.1362797& \
0.1375030 \\ \hline 
		\end{tabular}
		\vspace{0.2cm}
		\caption{Upper bounds for $\mathcal{C}^*(A)$ for each $A$, and the parameters to define the corresponding test functions for (EP2), defined as in \eqref{eq:psi-test}.}
		\label{tb:uppers}
	\end{center}
\end{table}
The computations were carried out in floating point arithmetic, using sufficient precision to justify the decimal digits shown in this section, in the sense that the digits shown in Table \ref{tb:uppers} remain stable after increasing precision. To carefully compute the numerical bound in (EP2) given the respective values of $A$ and $T_i$, we proceed as follows. We explain the computation in more detail in the case $A=1$, and the bounds for the other values of $A$ are verified similarly.

\smallskip

For $A\in\left\{\frac{1}{4},\frac{1}{3},\frac{1}{2},1,3\right\},$ consider the function 
\begin{equation}\label{eq:g_A}
 g_A(t):=\ 2\pi \left(
\,\widehat{{\bf 1}_{\R_-}\, e^{\pi (\cdot)}} - \widehat{{\bf 1}_{\R_+} \, \psi}
\right)(t),
\end{equation}
where $\psi$ is given by \eqref{eq:psi-test},  and $T_i$ are given by the corresponding values in Table \ref{tb:uppers}. Then, in (EP2) we have that $\mathcal{C}^*(A)\le \|g_A\|_\infty$. Let us compute $\|g_A\|_\infty$ in the case $A=1$. See Figure \ref{im:EP2_A1_long} for a plot of the test function $g_1$. One can check that $g_1(t)$ has three local maxima in the interval $[0,1.5]$, namely around $t=0$, $t=0.2902\ldots$, and $t=1.0410\ldots$; see Figure \ref{im:EP2_A1_short}. It is straightforward to numerically compute the maximum values to arbitrary precision for each of the three local maxima, and we find that the maximum of the three occurs at $t=0$, with the value $g_1(0)=1.1473077\ldots$. One can also clearly see that, for instance, $g_1(t)<1.1$ for $t>1.5$. Therefore, $\|g_1\|_\infty=1.1473077\ldots$, giving the upper bound shown in Table \ref{tb:uppers}. Similarly, we see that for $A=\frac{1}{4}$, $g_A$ has two local maxima near $t=0$ and $t=0.2287\ldots$ on the interval $[0,1]$ and is smaller beyond it, with the global maximum being near $t=0$ and giving the desired bound. For $A=\frac{1}{3}$, there are two local maxima near $t=0$ and $t=0.2648\ldots$ on the interval $[0,1]$, with the global maximum near $t=0$ giving the desired bound. For $A=\frac{1}{2}$, there are two local maxima near $t=0.0940\ldots$ and $t=0.5101\ldots$ on the interval $[0,1]$, with the global maximum near $t=0.5101\ldots$ giving the desired bound. Finally, for $A=3$, there are three local maxima near $t=0$, $t=0.3550\ldots$, and $t=2.1464\ldots$ on the interval $[0,4]$, and the maximum of these near $t=0$ gives the stated bound.

\begin{figure}[ht] 
	\centering
	\begin{minipage}{.5\textwidth}
		\centering
		\includegraphics[width=3in]{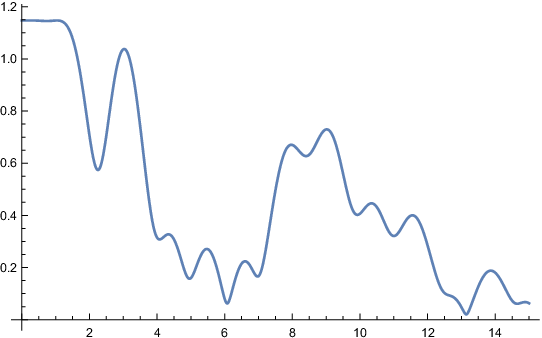} 
		\caption{\label{im:EP2_A1_long}The function  $g_1$ defined in \eqref{eq:g_A} on the interval $[0,15]$.
		}  
	\end{minipage}
	\begin{minipage}{.5\textwidth}
		\centering
		\includegraphics[width=.9\linewidth]{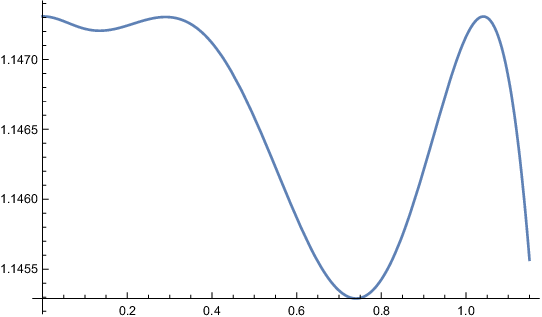} 
		\caption{\label{im:EP2_A1_short}The function  $g_1$ defined in \eqref{eq:g_A} on the interval $[0,1.1]$.}  
	\end{minipage}
\end{figure}

\subsection{Lower bound numerics}
Consider the Fourier transform pairs
\begin{equation}\label{20240411_14:28}
F_n(x)= \frac{1}{\pi(1 + 2ix)^{2n}}; \qquad \widehat{F_n}(t) = \frac{ -(\pi t)^{2n-1} \,e^{\pi t}\,{\bf 1}_{\R_-}(t)}{(2n-1)!},
\end{equation}
and note that $F_n(x)\in \mathcal{A}_\infty$ for all positive integers $n$.
In \S \ref{sec:c-infty}, we showed that the function $F_1(x)$ yields the extremal value $C(\infty)=1$. We now consider the following family of test functions, composed by dilations and translations of linear combinations of functions as in \eqref{20240411_14:28}, defined in terms of their Fourier transform:

\begin{equation}\label{eq:lowers}
 \widehat{F}(t)= g\left(\frac{\pi t - c}{a}
 \right)
 ,\text{ where } \ g(t)=\sum_{n=1}^N \frac{b_n\, t^{2n-1} \, e^{t}\, {\bf 1}_{\R_-}(t)}{(2n-1)!}. 
\end{equation}
Here, $b_n, c\in \R$ and $a>0$. Taken over this family of functions, (EP1) is an unrestricted optimization problem over $\R^{N+2}$ on the variables $b_1,\ldots,b_N, a, c$. The functional to maximize involves a numerical computation of the several integrals that appear in the formulation of (EP1) in \eqref{20230302_10:22}. It is also not smooth, since the integrands involve taking $L^1$-norms of complex-valued functions and positive and negative parts of real-valued functions. To optimize such a functional, we use the principal axis method of Brent \cite{Brent}, which searches for a local maximum of an unrestricted, non-smooth problem.

\smallskip

In Table \ref{tb:lowers}, we give our lower bound for $C(A)$ for each value of $A$, together with the parameters necessary to construct the test function as defined in \eqref{eq:lowers}. We normalize the coefficients $b_i$ so that, approximately, $\|F\|_1=1.0000.$ 
In Figure \ref{im:plots} we plot the functions $\widehat F(t)$, defined as in \eqref{eq:lowers} for the given values of $A$ and parameters in Table \ref{tb:lowers}.

\begin{table}[ht]
	\begin{center}
		\begin{tabular}{|c||c|c|c|c|c|}
		\hline
		$A$ & $\frac{1}{4}$ &$\frac{1}{3}$&$\frac{1}{2}$&$1$&$3$\\ \hline \hline
 Bound & 1.31706 & 1.27722 & 1.22112 &  1.14600 & 1.06082\\ \hline \hline
 $a$ & 0.856 & 0.727 & 0.587 & 0.246 & 0.209\\ \hline
$c$ & 1.082 & 0.922 & 0.758 & 0.626 & 0.201\\ \hline
$b_{1}$ & 0.071653 & 0.058845 & 0.037767 & -0.0027383 & 0.0079689\\ \hline
$b_{2}$ & 5.6943 & 5.0439 & 3.9929 & 0.0 & 2.2768\\ \hline
$b_{3}$ & -5.4722 & -3.2079 & 0.0 & 4.1716 & 0.27094\\ \hline
$b_{4}$ & 5.5724 & 3.0082 & 0.0 & -0.6464 & 4.7335\\ \hline
$b_{5}$ & -3.9261 & -1.4554 & 0.0 & 3.4098 & -4.076\\ \hline
$b_{6}$ & 1.3261 & -0.37314 & 0.0 & 1.0923 & 8.4977\\ \hline
$b_{7}$ & 0.33449 & 1.037 & 0.0 & -1.4628 & -7.0252\\ \hline
$b_{8}$ & -0.40291 & -0.5646 & 0.0 & 2.5377 & 5.9795\\ \hline
$b_{9}$ & 0.034651 & 0.12013 & 0.0 & 0.94904 & -1.1392\\ \hline
$b_{10}$ & 0.0 & -0.032923 & 0.0 & -2.5121 & -0.73835\\ \hline
$b_{11}$ & 0.0 & 0.0 & 0.0 & 2.1423 & 1.0537\\ \hline
$b_{12}$ & 0.0 & 0.0 & 0.0 & -0.3164 & 0.21283\\ \hline
		\end{tabular}
		\vspace{0.2cm}
		\caption{Lower bounds for $\mathcal{C}(A)$ for each $A$, and the parameters to define the corresponding test functions for (EP1), defined as in \eqref{eq:lowers}.}
		\label{tb:lowers}
	\end{center}
\end{table}

\begin{figure}[ht] 
	\centering
		\centering
		\includegraphics[width=5in]{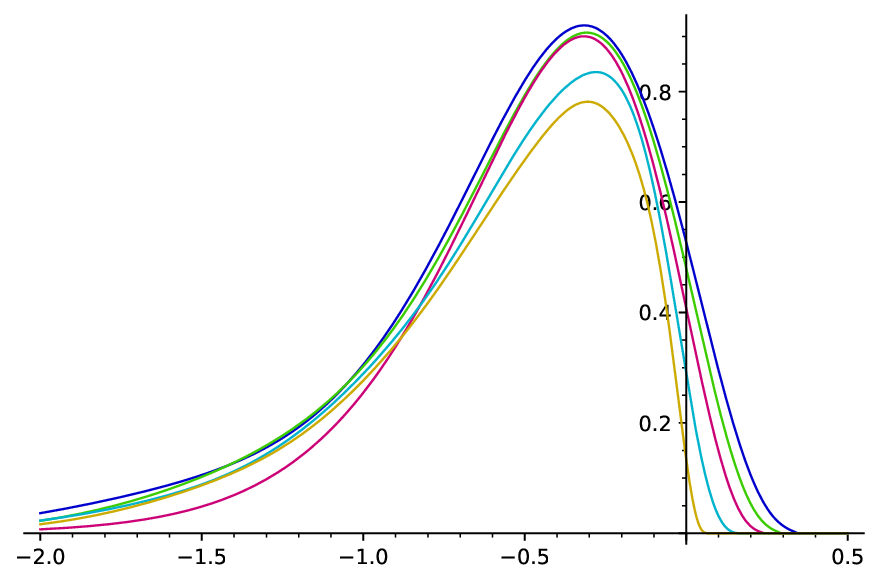}
		\caption{\label{im:plots}The functions  $\widehat F(t)$ defined in \eqref{eq:lowers} for each $A$, with the corresponding parameters from Table \ref{tb:lowers}. We have $A=3$ (yellow), $A=1$ (light blue), $A=1/2$ (magenta), $A=1/3$ (green), and $A=1/4$ (dark blue).
		}  
\end{figure}

To verify the lower bounds, we use the mpmath library \cite{mpmath} in Python 3.12.3, which permits numerical integration with arbitrary-precision floating-point arithmetic. In the arXiv version of this paper \cite{CMQHPR}, we attach Python code to verify the lower bounds with any desired working precision, for the given coefficients and values of $A$ in Table \ref{tb:lowers}. We use at least 30 decimal digits of working precision. The decimal digits for the lower bounds presented in this section are justified in the sense that they remain stable after successively increasing working precision to 100 decimal digits.

\smallskip

\subsection{Remarks on choosing a  family of test functions}
We make some heuristic remarks on the choice of family of test functions. The numerator in (EP1) is
\begin{equation*}
 \int_{-\infty}^{0} \widehat{F}(t)\,e^{\pi t} \,\d t - \int_{0}^{\infty} \widehat{F}_-(t)\,e^{\pi t} \,\d t - A\int_{0}^{\infty} \widehat{F}_+(t)\,e^{\pi t} \,\d t.
\end{equation*}
Thus, for a fixed value of $\|F\|_1$, the mass of $\widehat F$ should be concentrated on $\R_-$, to make the first term above large. The parameter $A$ acts as a penalty, and as $A\to\infty$, this concentration becomes more pronounced. We see this phenomena in Figure \ref{im:plots}, and in fact, for $A=\infty$, we know $F_1$ is optimal. Indeed, all functions $\widehat F_n$ in \eqref{eq:lowers} are continuous and supported in $\R_-$, making them good candidates after proper transformations. Moreover, note that $\|\pi \widehat F_n\|_1=\pi F_n(0)=1$ for all $n$. In practice, this normalization makes the coefficients in Table \ref{tb:lowers} have comparable sizes. Finally, if we include positive even powers of $t$ in \eqref{eq:lowers}, our experiments suggest the optimal functions have zero coefficients in such powers.  \smallskip

Other families were tested. A natural choice is a Hermite basis for $L^2(\R)$, composed of polynomials times Gaussians. Since this generates $L^2(\R)$, one might expect such a generic family to do well for polynomials of sufficiently high degree, and there are no a-priori assumptions on the shape of the resulting optimal functions. Nevertheless, in practice, this family performed significantly worse for the same dimensionality. Actually, this alternative choice of family seemed to slowly converge to the same general shape visible in Figure \ref{im:plots}. One might expect to require polynomials of very high degree to obtain the necessary concentration of mass described above, which might not be computationally viable: see, for instance, the discussion regarding choices of test functions in \cite[Section 7]{CQE}, in connection with the Fourier uncertainty principle.\smallskip

Recent works have used this Hermite family, combined with strong computational techniques-- in particular semidefinite programming-- to find numerical bounds for Fourier optimization problems associated to some number theoretic quantities of interest, using polynomials of high degree (see, for instance, \cite{CGL, CPL, QE}). The fact that we must necessarily work with complex-valued functions $F\in L^1(\R)$ is an important difference to the Fourier analysis frameworks in the aforementioned works (which use even, real-valued functions), and it is not clear if a similar approach with semidefinite programming may be applied in our situation.\smallskip

In Theorem \ref{Thm2}, note that the upper and lower bounds essentially coincide in the first two decimal digits, with our current choices of test functions. Since we have essentially arrived at the limit of the method, we do not attempt further numerical experimentation, leaving the problem of numerically identifying the third decimal digit, and beyond, to the interested reader.

\appendix

\section{The connection between $n_\chi$ and $S(t,\chi)$}

Assuming GRH, Montgomery \cite[Theorem 13.5]{Mont_thesis} has proved that there are real primitive characters $\chi$ modulo $q$ such that $n_\chi = \Omega(\log q \log\log q)$, as $q \to \infty$. Hence, on GRH, we have
\begin{equation}\label{20250805_16:52}
1 \le \limsup_{q \to \infty} \frac{\log n_\chi}{\log \log q} \le 2.
\end{equation}
Given this limitation, it is an interesting question to examine to what extent our bound for $n_\chi$ in Theorem \ref{Thm3} can improved, assuming conjectures beyond GRH. Inspired by the previous work of Montgomery \cite[Chapter 13]{Mont_thesis}, in this appendix we explore the connection between bounds for the least character non-residue $n_\chi$ and conjectural bounds for $S(t,\chi) = \frac{1}{\pi} \arg L(\tfrac{1}{2}+it,\chi)$. 

\smallskip

Montgomery observed that, assuming GRH, Ankeny's argument to prove that $n_\chi = O(\log^2q)$, as $q \to \infty$, relies on Selberg's bound that $S(t,\chi) = O(\frac{\log q\tau}{\log\log q\tau})$, where $\tau=|t|+2$. Furthermore, Montgomery made the connection between $n_\chi$ and $S(t,\chi)$ explicit. Assuming GRH and letting 
\[
M(\chi) = \int_{-\infty}^\infty |S(t,\chi)| \, e^{-|t|} \, \d t, 
\]
Montgomery \cite[Corollary 13.4]{Mont_thesis} proved that
\begin{equation} \label{Mont_bound}
n_\chi \ll \big( M(\chi) \log\log q \big)^2 + \sqrt{\log q}.
\end{equation}
From this result and Selberg's bound $S(t,\chi)= O(\frac{\log q\tau}{\log\log q\tau})$, it immediately follows that $n_\chi = O(\log^2q)$.  Moreover, from Montgomery's result, it is clear that if one can improve the order of magnitude for the upper bound of $|S(t,\chi)|$ in the $q$-aspect, then one can deduce a corresponding order of magnitude improvement for the upper bound for $n_\chi$ as function of $q$. 

\smallskip

Inspired by Montgomery's result in \eqref{Mont_bound}, we prove the following theorem.

\begin{theorem}\label{arg_bound}
Assume GRH for Dirichlet $L$-functions and let $\mc{M}$ be any function such that
\begin{equation} \label{S_assump1}
|S(t,\psi)| \le \big( 1+o(1)\big) \cdot \mc{M}(k) + O(\log \tau)
\end{equation}
for primitive characters $\psi$ modulo $k$, as $k \to \infty$, where $\tau=|t|+2$. Then, if $\chi$ is a non-principal character modulo $q$ of order $\ell$,  we have 
\[
\limsup_{q \to \infty} \frac{n_\chi}{  (\mc{M}(q) \log n_\chi)^2 } \le \left(\frac{2\pi}{\mc{C}\big(\frac{1}{\ell\!-\!1}\big)}\right)^{2},
\]
where $\mc{C}(A)$ is the constant defined in \eqref{20230302_10:22}.
\end{theorem}

Assuming GRH, Montgomery \cite[Corollary 13.6]{Mont_thesis} has shown that there are pairs of  real primitive characters $\chi$ and real numbers $t$ with 
\[
|S(t,\chi)| \gg \sqrt{ \frac{\log q}{\log \log q}},
\]
where $q$ is arbitrarily large and $|t| \ll \log \log q$. This shows that 
\begin{equation} \label{M_bound}
\mc{M}(q)\gg \sqrt{ \frac{\log q}{\log \log q}}.
\end{equation}
Assuming GRH, the sharpest known bound for $|S(t,\chi)|$  is due to Carneiro, Chandee, and Milinovich \cite[Theorem 5 and Example 8]{CCM2} who proved that $\mc{M}(q) = \frac{1}{4} \frac{\log q}{ \log\log q}$ is admissible in \eqref{S_assump1}. Using this estimate, together with \eqref{20250805_16:52}, in the above theorem leads to a bound for $n_\chi$ of the form
\[
\limsup_{q \to \infty} \frac{n_\chi}{ \log^2q } \le \left(\frac{\pi}{\mc{C}\big(\frac{1}{\ell\!-\!1}\big)}\right)^{2},
\]
which is weaker than Theorem \ref{Thm3}  by a factor of $\pi^2$. If one could prove that there is constant $c < \frac{1}{4 \pi}$ such that $\mc{M}(q) = c \, \frac{\log q}{ \log\log q}$ is admissible in \eqref{S_assump1}, then Theorem \ref{arg_bound} could be used to improve upon Theorem \ref{Thm3}. It is widely believed that this is possible and, in fact, that $\mc{M}(q)$ can be chosen to be a considerably slower growing function of $q$ than what is proved in \cite{CCM2}. The work of Farmer, Gonek, and Hughes \cite{FGH} suggests that the following conjecture holds. 

\begin{conjecture}\label{S_assump2}
Let $\chi$ be a primitive character modulo $q$. As  $q \to \infty$, there is a constant $B$ such that
\begin{equation}\label{A}
|S(t,\chi)| \le \big( B+o(1)\big) \sqrt{ \log q \log \log q \,} + O(\log \tau),
\end{equation}
where $\tau=|t|+2$.
\end{conjecture}

Assuming this conjecture, we deduce the following corollary of Theorem \ref{arg_bound}.

\begin{corollary}\label{S_cor}
Let $\chi$ be a non-principal character modulo $q$ of order $\ell$ and assume both GRH for Dirichlet $L$-functions and Conjecture \ref{S_assump2}. Then
\[
\limsup_{q \to \infty} \frac{n_\chi}{  \log q \, (\log \log q)^3 } \le \left(\frac{2\pi B}{\mc{C}\big(\frac{1}{\ell\!-\!1}\big)}\right)^{2}, 
\]
where $B$ is the constant appearing in \eqref{A} and $\mc{C}(A)$ is the constant defined in \eqref{20230302_10:22}.
\end{corollary}
\begin{proof}[Proof of Corollary \ref{S_cor}]
Assuming GRH for $L(s,\chi)$ and Conjecture \ref{S_assump2}, it follows from Montgomery's bound in \eqref{Mont_bound} that $n_\chi \ll \log q \, (\log\log q)^3$, and hence that $\log n_\chi \le \big(1+o(1)\big) \log\log q$, as $q\to\infty$. Combining this observation with the fact that $n_\chi = \Omega(\log q \log\log q)$, as $q \to \infty$, we see that 
\begin{equation} \label{A2}
\limsup_{q \to \infty} \frac{n_\chi}{  (\mc{M}(q) \log n_\chi)^2 } = \limsup_{q \to \infty} \frac{n_\chi}{  (\mc{M}(q) \log \log q)^2 },
\end{equation}
under the assumptions of Corollary \ref{S_cor}. Now observe that Conjecture \ref{S_assump2} implies that the function
\[
\mc{M}(q) = B \sqrt{\log q \log \log q \,}
\] 
is admissible in \eqref{S_assump1}, so the result of
Theorem \ref{arg_bound} combined with \eqref{A2} implies that
\[
\limsup_{q \to \infty} \frac{n_\chi}{ B^2 \, \log q \, (\log \log q)^3 } 
=  \limsup_{q \to \infty} \frac{n_\chi}{  (\mc{M}(q) \log n_\chi)^2 }  \le  \left(\frac{2\pi }{\mc{C}\big(\frac{1}{\ell\!-\!1}\big)}\right)^{2}.
\]
Multiplying this expression by $B^2$, we deduce the corollary. 
\end{proof}

It remains to prove Theorem \ref{arg_bound}. 

\begin{proof}[Proof of Theorem \ref{arg_bound}] We proceed exactly as we did in the proof of Theorem \ref{Thm3}, only we estimate the contribution from the zeros of the relevant $L$-functions differently. As before, we let $\chi$ be a non-principal character modulo $q$ of order $\ell$, and set $2\pi \Delta:= \log n_\chi$. We again let $F \in \mc{A}$ with $\widehat{F} \in C^{\infty}_c(\R)$ be fixed, with $\supp\big( \widehat{F}\big) \subset [-N, N]$, and let
\[
\widehat{h}(t) := \frac{\widehat{F}(t - \Delta) + \widehat{F}(-t - \Delta)}{2}.
\]
Again we assume, without loss of generality, that $\Delta \geq N$. Thus $\widehat{h}(0)=0$, and we can combine estimates \eqref{20230323_10:21} - \eqref{20230327_13:19} to deduce that
\begin{equation} \label{A1}
\begin{split}
&\left|  \sum_{j=1}^{\ell-1} \sum_{\gamma_{(\chi^j)^*}} h\!\left(\gamma_{(\chi^j)^*}\right)   \right| 
\\
& \quad \ge
e^{\pi \Delta}(\ell \!- \!1) \left(   \int_{-\infty}^{0}\widehat{F}(y) \, e^{\pi y} \,\d y -   \int_{0}^{\infty}\widehat{F}_-(y)\,  e^{\pi y} \,\d y - \frac{1}{(\ell\!-\!1)} \int_{0}^{\infty}\widehat{F}_+(y) \, e^{\pi y} \,\d y \right)+ O_F\!\left( \ell\sqrt{\log q}\right),
\end{split}
\end{equation}
where $(\chi^j)^*$ denotes the primitive character inducing $\chi^j$ and $\gamma_{(\chi^j)^*}$ denotes an ordinate of a non-trivial zero of $L(s,(\chi^j)^*)$. 

\smallskip

Now, for our choice of $h$ and any primitive character $\psi$ modulo $k$, using the notation in the proof of Lemma \ref{lemSumOverZeros} and \eqref{20230327_13:02}, we have
\[
\begin{split}
\sum_{\gamma_{\psi}} h\!\left(\gamma_{\psi}\right) &=\int_{0^-}^{\infty} h(t) \, \d N(t, \psi) +  \int_{0^-}^{\infty} h(-t) \, \d N(t, \overline{\psi})
\\
&= \frac{\widehat{h}(0) \log q}{2 \pi} + O\left( \int_0^\infty |h(t)| \log(t\!+\!2) \, \d t \right) + \int_{0^-}^{\infty} h(t) \, \d S(t, \psi) +  \int_{0^-}^{\infty} h(-t) \, \d S(t, \overline{\psi})
\\
&= \int_{0^-}^{\infty} h(t) \, \d S(t, \psi) +  \int_{0^-}^{\infty} h(-t) \, \d S(t, \overline{\psi}) + O_F(1),
\end{split}
\]
upon recalling that $\widehat{h}(0)=0$. Integrating by parts and using \eqref{S_assump1}, we have
\[
\int_{0^-}^{\infty} h(t) \, \d S(t, \psi) +  \int_{0^-}^{\infty} h(-t) \, \d S(t, \overline{\psi}) =  -\int_{0}^{\infty} h'(t) \, S(t, \psi) \, \d t +  \int_{0}^{\infty} h'(-t) \, S(t, \overline{\psi})\,  \d t  + O_F\big( \mc{M}(k) \big).
\]
Since 
\[
h(t) = \tfrac12\big(e^{2\pi i t \Delta} F(t) + e^{-2 \pi i t \Delta}F(-t)\big),
\]
again using \eqref{S_assump1}, we have 
\[
\begin{split}
&\left| \int_{0}^{\infty} h'(t) \, S(t, \psi) \, \d t \right| 
\\
&\qquad \qquad \le 2 \pi \Delta \int_0^\infty \Bigg(\frac{|F(t)| + |F(-t)|}{2}\Bigg) \, |S(t, \psi)| \, \d t + \int_0^\infty \Bigg(\frac{|F'(t)| + |F'(-t)|}{2}\Bigg) \, |S(t, \psi)| \, \d t
\\
&\qquad \qquad \le 2 \pi \Delta \Big(\mc{M}(k) \!+\! o(1) \Big) \frac{ ||F||_1 }{2}   + O_F\big( \mc{M}(k) \big),
\end{split}
\] 
as $k\to \infty$. Similarly, we have
\[
\left|\int_{0}^{\infty} h'(-t) \, S(t, \overline{\psi})\,  \d t  \right| \le 2 \pi \Delta \Big(\mc{M}(k) \!+\! o(1) \Big) \frac{ ||F||_1 }{2} + O_F\big( \mc{M}(k) \big)
\]
and therefore
\begin{equation}\label{zero_bound}
\left|\sum_{\gamma_{\psi}} h\!\left(\gamma_{\psi}\right) \right| \le 2 \pi \Delta \Big(\mc{M}(k) \!+\! o(1) \Big) ||F||_1 + O_F\big( \mc{M}(k) \big).
\end{equation}

\smallskip

To complete the proof, we use the bound \eqref{zero_bound} on each sum over zeros to deduce that
\begin{equation} \label{zeros_ineq}
\begin{split}
\left|  \sum_{j=1}^{\ell-1} \sum_{\gamma_{(\chi^j)^*}} h\!\left(\gamma_{(\chi^j)^*}\right)   \right| &\le \sum_{j=1}^{\ell-1} \left|  \sum_{\gamma_{(\chi^j)^*}} h\!\left(\gamma_{(\chi^j)^*}\right)   \right|
\\
&\le (\ell-1) \, 2 \pi \Delta  \Big(\mc{M}(q) \!+\! o(1) \Big) ||F||_1 + O_F\big( \mc{M}(q) \big)
\end{split}
\end{equation}
Here we have used the fact that the modulus of each $(\chi^j)^*$ is less than or equal to $q$. Combining the inequalities in \eqref{A1} and \eqref{zeros_ineq} and dividing by $\ell-1$, it follows that

\begin{align*}
e^{\pi \Delta} &\left(   \int_{-\infty}^{0}\widehat{F}(y) \, e^{\pi y} \,\d y -   \int_{0}^{\infty}\widehat{F}_-(y)  \,e^{\pi y} \,\d y - \frac{1}{(\ell\!-\!1)} \int_{0}^{\infty}\widehat{F}_+(y)  \,e^{\pi y} \,\d y \right) 
\\
& \qquad \qquad \qquad \qquad  \qquad \qquad  \leq   2 \pi \Delta  \Big(\mc{M}(q) \!+\! o(1) \Big) ||F||_1 + O_F\big( \mc{M}(q) \big) + O_F\!\left(\sqrt{\log q}\right).
\end{align*}
Recalling that $n_{\chi} = e^{2\pi \Delta}$ and using \eqref{M_bound}, this yields
\begin{align}\label{final_form}
\limsup_{q \to \infty}\frac{\sqrt{n_{\chi}}}{\mc{M}(q)  \log n_\chi} \leq  \frac{\|F\|_1}{ \left(   \int_{-\infty}^{0}\widehat{F}(y) \, e^{\pi y} \,\d y -   \int_{0}^{\infty}\widehat{F}_-(y) \, e^{\pi y} \,\d y - \frac{1}{(\ell-1)} \int_{0}^{\infty}\widehat{F}_+(y) \, e^{\pi y} \,\d y \right)}\,,
\end{align}
where we assume that the denominator on the right-hand side of \eqref{final_form} is positive. At this stage we can take the infimum of the right-hand side of \eqref{final_form} over $F \in \mc{A}$ with $\widehat{F} \in C^{\infty}_c(\R)$ and, by Theorem \ref{Thm1} (i), such an infimum is indeed $2\pi \, \mc{C} \! \left(  \frac{1}{(\ell-1)}\right)^{-1}$. This concludes the proof of Theorem \ref{arg_bound}.
\end{proof}

\section*{Acknowledgments}
Part of this research took place while MBM and EQH were visiting the Abdus Salam International Centre for Theoretical Physics (ICTP). They thank ICTP for the hospitality. MBM was supported by the NSF grants DMS-2101912 and DMS-2401461, and the Simons Foundation (award 712898). EQH was supported by the Austrian Science Fund (FWF), projects P-34763 and P-35322, and a PIMS-Simons postdoctoral fellowship at the University of Lethbridge.


\begin{thebibliography}{99}

\bibitem{An}
N. C. Ankeny,
\newblock The least quadratic non residue,
\newblock Ann. of Math. (2) 55 (1952), 65--72. 

\bibitem{Bach}
E. Bach, 
\newblock Explicit bounds for primality testing and related problems, 
\newblock Math. Comp. 55 (1990), no. 191, 355--380. 

\bibitem{BachS}
E. Bach and J. Sorenson, 
\newblock Explicit bounds for primes in residue classes, 
\newblock Math. Comp. 65 (1996), no. 216, 1717--1735.

\bibitem{BRS}
A. Bondarenko, D. Radchenko and K. Seip,
\newblock Fourier interpolation with zeros of zeta and L-functions,
\newblock Constr. Approx. 57 (2023), no. 2, 405--461.

\bibitem{Brent}
R. P. Brent, 
\newblock {\it Algorithms for minimization without derivatives}, 
\newblock Prentice-Hall Series in Automatic Computation. Prentice-
Hall, Inc., Englewood Cliffs, N.J., 1973.

\bibitem{BGMM}
H.~M.~Bui, D.~A.~Goldston, M. B. Milinovich, and H. L. Montgomery,
\newblock Small gaps and small spacings between zeta zeros,
\newblock Acta Arith. 210 (2023), 133--153.

\bibitem{CCCM}
E. Carneiro, V. Chandee, A. Chirre, and M. B. Milinovich,
\newblock On Montgomery's pair correlation conjecture: a tale of three integrals,
\newblock J. Reine Angew. Math. 786 (2022), 205--243. 

\bibitem{CCLM}
E. Carneiro, V. Chandee, F. Littmann, and M. B. Milinovich,
\newblock Hilbert spaces and the pair correlation of zeros of the Riemann zeta-function,
\newblock J. Reine Angew. Math. 725 (2017), 143--182. 

\bibitem{CCM}
E. Carneiro, V. Chandee, and M. B. Milinovich, 
\newblock Bounding $S(t)$ and $S_1(t)$ on the Riemann hypothesis,
\newblock Math. Ann. 356 (2013), no. 3, 939--968.

\bibitem{CCM2}
E. Carneiro, V. Chandee, and M. B. Milinovich, 
\newblock A note on the zeros of zeta and $L$-functions,
\newblock Math. Z. 281 (2015), no. 1-2, 315--332.

\bibitem{CChi} 
E. Carneiro and A. Chirre,
\newblock Bounding $S_n(t)$ on the Riemann hypothesis,
\newblock Math. Proc. Cambridge Philos. Soc. 164 (2018), no. 2, 259--283.

\bibitem{CChiM}
E. Carneiro, A. Chirre, and M. B. Milinovich, 
\newblock Hilbert spaces and low-lying zeros of $L$-functions,
\newblock Adv. Math. 410 (2022), part B, Paper No. 108748, 48 pp.

\bibitem{CarFinder}
E. Carneiro and R. Finder, 
\newblock On the argument of L-functions,
\newblock Bull. Braz. Math. Soc. (N.S.) 46 (2015), no. 4, 601--620. 

\bibitem{CMQHPR}
E. Carneiro, M. B. Milinovich, E. Quesada-Herrera, and A. P. Ramos,
\newblock Fourier optimization, the least quadratic non-residue, and the least prime in an arithmetic progression, 
\newblock arXiv preprint: https://arxiv.org/abs/2404.08380.

\bibitem{CMPR}
E. Carneiro, M. B. Milinovich, and A. P. Ramos,
\newblock Fourier optimization and Montgomery's pair correlation conjecture, 
\newblock Math. Comp. 94 (2025), no. 351, 409--424.

\bibitem{CMS}
E. Carneiro, M. B. Milinovich, and K. Soundararajan, 
\newblock Fourier optimization and prime gaps,
\newblock Comment. Math. Helv. 94 (2019), no. 3, 533--568.

\bibitem{CS} 
V. Chandee and K. Soundararajan,
\newblock Bounding $|\zeta(\frac{1}{2}+it)|$ on the Riemann hypothesis, 
\newblock Bull. London Math. Soc. 43 (2011), no. 2, 243--250.

\bibitem{CGL} 
A. Chirre, F. Gon\c{c}alves, and D. de Laat,
\newblock Pair correlation estimates for the zeros of the zeta function via semidefinite programming, 
\newblock Adv. Math. 361 (2020), 106926, 22.

\bibitem{CPL} 
A. Chirre, V. J. Pereira J\'{u}nior, and D. de Laat,
\newblock Primes in arithmetic progressions and semidefinite programming, 
\newblock Math. Comp. 90 (2021), no. 331, 2235–2246.

\bibitem{CQE}
A. Chirre and E. Quesada-Herrera, 
\newblock Fourier optimization and quadratic forms,
\newblock Q. J. Math. 73 (2022), no. 2, 539--577. 

\bibitem{CE}
H. Cohn and N. Elkies,
\newblock New upper bounds on sphere packings. I.
\newblock Ann. of Math. (2) 157 (2003), no. 2, 689--714.

\bibitem{CKMRV}
H. Cohn, A. Kumar, S. D. Miller, D. Radchenko, and M. Viazovska,
\newblock The sphere packing problem in dimension 24,
\newblock Ann. Math. 185 (2017), 1017--1033.

\bibitem{CKMRV2}
H. Cohn, A. Kumar, S. D. Miller, D. Radchenko, and M. Viazovska,
\newblock Universal optimality of the $E_8$ and Leech lattices and interpolation formulas,
\newblock Ann. of Math. (2) 196 (2022), no. 3, 983--1082.

\bibitem{FGH}
D.~W.~Farmer, S.~M.~Gonek, and C.~P.~Hughes, 
\newblock The maximum size of $L$-functions,
\newblock J. Reine Angew. Math. 609 (2007), 215--236.

\bibitem{FM}
J.~Freeman and S.~J.~Miller,
\newblock Determining optimal test functions for bounding the average rank in families of $L$-functions, 
\newblock SCHOLAR -- a scientific celebration highlighting open lines of arithmetic research, 97--116, Contemp. Math., 655, Centre Rech.~Math.~Proc., Amer.~Math.~Soc., Providence, RI, 2015. 

\bibitem{Ga}
P.~X.~Gallagher, 
\newblock Pair correlation of zeros of the zeta function, 
\newblock J. Reine Angew. Math. 362 (1985), 72--86.

\bibitem{Go}
D. V. Gorbachev, 
\newblock Extremal problem for entire functions of exponential spherical type, connected with the Levenshtein bound on the sphere packing density in $\R^n$ (Russian), 
\newblock Izvestiya of the Tula State University Ser. Mathematics Mechanics Informatics 6 (2000), 71--78.

\bibitem{HB}
D. R. Heath-Brown, 
\newblock Zero-free regions for Dirichlet L-functions, and the least prime in an arithmetic progression, 
\newblock Proc. London Math. Soc. (3) 64 (1992), no. 2, 265--338.

\bibitem{IK} 
H. Iwaniec and E. Kowalski, 
\newblock {\it Analytic Number Theory}, 
\newblock AMS Colloquium Publications, vol. 53 (2004).

\bibitem{LLS}
Y. Lamzouri, X. Li, and K. Soundararajan, 
\newblock Conditional bounds for the least quadratic non-residue and related problems,
\newblock Math. Comp. 84 (2015), no. 295, 2391--2412. 

\bibitem{LLS2}
Y. Lamzouri, X. Li, and K. Soundararajan, 
\newblock Corrigendum to ``Conditional bounds for the least quadratic non-residue and related problems'',
\newblock Math. Comp. 86 (2017), no. 307, 2551--2554. 

\bibitem{Mont_thesis}
H. L. Montgomery,
\newblock {\it Topics in multiplicative number theory},
\newblock Lecture Notes in Math., Vol. 227, Springer-Verlag, Berlin-New York, 1971, ix+178 pp.

\bibitem{Mont}
H. L. Montgomery,
\newblock Distribution of the zeros of the Riemann zeta function,
\newblock Proceedings of the International Congress of Mathematicians (Vancouver, B. C., 1974), Vol. 1, pp 379--381.
Canad. Math. Congress, Montreal, Que., 1975

\bibitem{mpmath}
The mpmath development team,
\newblock mpmath: a Python library for arbitrary-precision floating-point arithmetic (version 1.3.0), (2023):
\newblock http://mpmath.org/ 

\bibitem{MV_Book}
H. L. Montgomery and R. C. Vaughan, 
\newblock {\it Multiplicative number theory. I. Classical theory.}
\newblock Cambridge Studies in Advanced Mathematics, 97. Cambridge University Press, Cambridge, 2007.

\bibitem{MV_LS}
H. L. Montgomery and R. C. Vaughan, 
\newblock The large sieve,
\newblock Mathematika 20 (1973), 119--134. 


\bibitem{ILS}
H. Iwaniec, W. Luo, and P. Sarnak, 
\newblock Low-lying zeros of families of $L$-functions, 
\newblock Publ. Math. Inst. Hautes Études Sci. 91 (2000) 55--131.

\bibitem{QE}
E. Quesada-Herrera,  
\newblock On the q-analogue of the pair correlation conjecture via Fourier optimization,
\newblock Math. Comp. 91 (2022), no. 337, 2347--2365.

\bibitem{RS}
M. C. Reed and B. Simon, 
\newblock {\it Methods of modern mathematical physics. II. Fourier analysis, self-adjointness},
\newblock Academic Press, New York-London, 1975.

\bibitem{Selberg}
A. Selberg, 
\newblock Contributions to the theory of Dirichlet's $L$-functions,
\newblock Skr. Norske Vid.-Akad. Oslo I 1946 (1946), no. 3, 62 pp.

\bibitem{Sono}
K. Sono, 
A note on simple zeros of primitive Dirichlet $L$-functions, 
\newblock Bull. Aust. Math. Soc. 93 (2016), no. 1, 19--30.

\bibitem{Vi}
M. Viazovska,
\newblock The sphere packing problem in dimension 8,
\newblock Ann. Math. 185 (2017), 991--1015.

\bibitem{LV}
A. I. Vinogradov and Ju. V. Linnik, 
\newblock Hypoelliptic curves and the least prime quadratic residue, 
\newblock Dokl. Akad. Nauk SSSR., 168 (1966), 258--261.

\bibitem{X}
T. Xylouris, 
\newblock On the least prime in an arithmetic progression and estimates for the zeros of Dirichlet L-functions, 
\newblock Acta Arith. 150 (2011), 65--91.






\end{thebibliography}
\end{document}